%% file: 2D_numerical_solutions.tex
\documentclass[11pt,a4paper,reqno,final]{article}
\usepackage{graphicx}
\usepackage{amsmath}
\usepackage[margin=1in]{geometry}
\usepackage{authblk}
\usepackage{mathtools}
\usepackage{esint}

\newcommand{\vx}{\mathrm{d}{\boldsymbol x}}

\newcommand{\velc}{{\boldsymbol u}}

\newcommand{\velcex}{{\widetilde{\boldsymbol u}}}

\newcommand{\alex}{\widetilde{\alpha}}

\newcommand{\vecfd}{{\boldsymbol F}}
\newcommand{\Cex}{\widetilde{c}}

\newcommand{\hn}{\mathbf{H}_d^1(\Omega(t))}
\newcommand{\hna}{\mathbf{H}_d^1(A)}

\newcommand{\uspace}{H_{\nabla}^{1,u}(D_T)}
\newcommand{\cspace}{H_{\nabla}^{1,c}(D_T)}

\newcommand{\utanspace}{\boldsymbol{\mathrm H}^1_{0,\boldsymbol{\tau}}(\Omega(t))}
\newcommand{\upspace}[1]{\boldsymbol{\mathrm{H}}_{u,p}^{#1}}
\newcommand{\utanspacea}{\boldsymbol{\mathrm H}^1_{0,\boldsymbol{\tau}}(A)}
\newcommand{\stboundary}{\partial D_{T}\backslash ((\{0\}\times\Omega(0))\cup (\{T\}\times\Omega(T)))}

\newcommand{\rigid}{$B{\boldsymbol x} + {\boldsymbol \beta}$}
\newcommand{\astar}{\alpha^{\ast}}
\newcommand{\athr}{\alpha_{\mathrm{thr}}}

\usepackage{multirow}
\usepackage[table,xcdraw]{xcolor}

\usepackage{amsmath,amsthm,amssymb}
\usepackage{mathrsfs}
\numberwithin{equation}{section}
\usepackage[labelformat=simple]{subcaption}

\usepackage{enumitem}
\usepackage{mathtools}  
\mathtoolsset{showonlyrefs} 

\theoremstyle{plain}
\newtheorem{theorem}{Theorem}[section]

\theoremstyle{plain}
\newtheorem{proposition}[theorem]{Proposition}
\newtheorem{definition}[theorem]{Definition}

\newtheorem{lemma}[theorem]{Lemma}
\newtheorem{remark}[theorem]{Remark}

\newenvironment{mproof}{\emph{Proof.}\;}{}

\begin{document}

\title{Numerical solution of a two dimensional tumour growth model with moving boundary}
\author[1]{J\'{e}r\^{o}me Droniou\thanks{Email: jerome.droniou@monash.edu}}
\author[2]{Jennifer. A. Flegg\thanks{Email: jennifer.flegg@unimelb.edu.au}}
\author[3]{Gopikrishnan. C. Remesan\thanks{Email: gopikrishnan.chirappurathuremesan@monash.edu}}
\date{\today}
\affil[1]{\small{School of Mathematics, Monash University, Victoria 3800, Australia}}
\affil[2]{\small{School of Mathematics and Statistics, University of Melbourne, Melbourne, Australia}}
\affil[3]{\small{IITB - Monash Research Academy, Indian Institute of Technology Bombay, Powai, India}}
\maketitle
\begin{abstract}
	We consider a biphasic continuum model for avascular tumour growth in two spatial dimensions,  in which a cell phase and a fluid phase follow conservation of mass and momentum. A limiting nutrient that follows a diffusion process controls the birth and death rate of the tumour cells.  The cell volume fraction, cell velocity--fluid pressure system, and nutrient concentration are the model variables. A coupled system of a hyperbolic conservation law, a viscous fluid model, and a parabolic diffusion equation governs the dynamics of the model variables.  The tumour boundary moves with the normal velocity of the outermost layer of cells, and this time--dependence is a challenge in designing and implementing a stable and fast numerical scheme.  We recast the model into a form where the hyperbolic equation is defined on a fixed extended domain and retrieve the tumour boundary as the interface at which the cell volume fraction decreases below a threshold value.  This procedure eliminates the need to track the tumour boundary explicitly and the computationally expensive re--meshing of the time--dependent domains.  A numerical scheme based on finite volume methods for the hyperbolic conservation law, Lagrange $\mathbb{P}_2 - \mathbb{P}_1$ Taylor--Hood finite element method for the viscous system, and mass--lumped finite element method for the parabolic equations is implemented in two spatial dimensions, and several cases are studied. We demonstrate the versatility of the numerical scheme in catering for irregular and asymmetric initial tumour geometries. When the nutrient diffusion equation is defined only in the tumour region, the model depicts growth in free suspension. On the contrary, when the nutrient diffusion equation is defined in a larger fixed domain, the model depicts tumour growth in a polymeric gel. We present numerical simulations for both cases and the results are consistent with theoretical and heuristic expectations such as early linear growth rate and preservation of radial symmetry when the boundary conditions are symmetric. The work presented here could be extended to include the effect of drug treatment of growing tumours.
\end{abstract}



\section{Introduction}
The initial growth of a proliferating tumour does not contain vascular tissues, which forces the tumour to depend on diffused nutrients from the surrounding environment for its growth. The modelling and numerical simulations of this stage, namely the avascular growth stage, has been a frontier research area since the late 1970s~\cite{Greenspan1976229, ward_1, ward_2}. Depending on the scale of observation -- cellular level (microscopic) or aggregate level (tissue or macroscopic) -- and nature of interactions between the constituents, there are several mathematical approaches and methods to model the avascular growth stage. A detailed review of various models can be found in Roose et al.~\cite{roose_et_al} and Araujo et al.~\cite{Araujo20041039}.

An extensive amount of scientific literature is available regarding the mathematical modelling of avascular tumour growth and multicellular spheroids~\cite{breward_2002,breward_2003,byrne_chaplain_1997,byrne_mcelawin_preziosi_2003,Byrne_prezziozi_2003,scuime}. We focus on models based on mass balance equations, diffusion equations, and continuum mechanics~\cite{Osborne20103402}. Such models are reasonably easy to numerically implement  using appropriate combinations of  finite element methods and finite volume methods. This paper complements the previously mentioned works by relaxing several assumptions and extending to more general situations like asymmetric and irregular initial tumour geometries.

We consider a biphasic and viscous tumour model with a time--dependent spatial boundary in two and three spatial dimensions. The tumour cells constitute a viscous phase called the \emph{cell phase} and the surrounding fluid medium constitute an inviscid phase called the \emph{fluid phase}. The cell and fluid phases actively exchange matter through the processes of cell division and cell death. The diffusing nutrient controls the birth and death rates of the cells. H. M. Byrne et al.~\cite{byrne_mcelawin_preziosi_2003}  considered an early version of this model and C. J. W. Breward et al.~\cite{breward_2002,breward_2003} conducted a detailed study of the one--dimensional version. In these works, the authors present a detailed analysis of the effect of model parameters including the viscosity coefficient of the cell phase, drag coefficient between the cell and fluid phases, and parameters that determine attractive and repulsive forces between the tumour cells. A model based on multiphase mixture theory is described in the work by H. M. Byrne and L. Preziosi~\cite{Byrne_prezziozi_2003}, in which they use a continuous cell--cell force term in contrast to the discontinuous force term in~\cite{breward_2002}. 

The previously mentioned models successfully describe the evolution of tumour radius and the effect of  model parameters. However, to reduce a higher spatial dimensional model to a single spatial dimension, it is assumed that the tumour is growing radially symmetrically. This assumption is not valid if the initially seeded tumour is irregular in shape. Also, the time--dependent boundary is not well defined except in the radially symmetric case.  In this article, we adapt and recast the model in~\cite{byrne_mcelawin_preziosi_2003} such that symmetry assumptions are relaxed, ill--posedness of the time--dependent boundary is  corrected, and numerical simulations are feasible without reducing the dimensionality. 

J. M. Osborne and J. P. Whiteley~\cite{Osborne20103402} developed a generic numerical framework for multiphase viscous flow equations and applied it to simulate tissue engineering models and tumour growth models. Though the numerical scheme presented in~\cite{Osborne20103402} is robust, the tumour growth model considered is ill--posed. Here, the viscous system that governs the cell velocity has a solution unique only up to a (rigid--body motion) function of the form $\boldsymbol{u}(\boldsymbol{x})$ = \rigid, where $B$ is a skew--symmetric matrix, ${\boldsymbol x} \in \mathbb{R}^d$, and ${\boldsymbol
	\beta} \in \mathbb{R}^d$ is a constant. This non--uniqueness for viscous equations with pure traction boundary condition is a well--established fact in the theory of continuum mechanics~\cite[p.~155]{alexander}. At the discrete level, the resulting non--invertibility of the coefficient matrix is overcome by imposing an auxiliary condition. A natural approach is to set the cell velocity at the centre of the tumour to be zero. However, this approach has the following drawbacks. Firstly, the auxiliary condition is not inbuilt with the model; instead, it is a numerical level fix. Secondly, in the case of an asymmetrically shaped tumour a well--defined centre is absent. Even if  we define the centre in a mathematical way, say as the centre of mass, it will vary over time, and consequently, the auxiliary condition as well, thereby making the numerical algorithm computationally intense. Thirdly, fixing the velocity at a single point does not fully eliminate the non--uniqueness. In fact, in two dimensions, even after imposing this condition, solution of the viscous equation is unique only up to functions of the form ${\boldsymbol u}(x,y) = a (y_0-y,x-x_0) + (\alpha_1,\alpha_2)$, where $a \in \mathbb{R}$ is an arbitrary constant and ${\boldsymbol u}(x_0,y_0) = (\alpha_1,\alpha_2)$ for fixed vectors $(x_0,y_0)$ and $(\alpha_1, \alpha_2)$. The function $\boldsymbol{u}$ can be decomposed into the form, $\boldsymbol{u}(x,y) = aB_{\pi/2}(x,y)^T + (ay_0 + \alpha_1, -ax_0 + \alpha_2)$, where the matrix $B_{\pi/2} = \begin{psmallmatrix}
0 & -1 \\1 & 0
\end{psmallmatrix}$ represents the anticlockwise rotation by $\pi/2$ radians. Therefore, $\boldsymbol{u}$ is the sum of a scaled rotation and a translation in the Cartesian plane, and such functions constitute the null space of the linear operator acting on $\boldsymbol{u}$.  In the current work, we circumvent the need for any such numerical fix by ensuring the well--posedness of the viscous system. In particular, we employ appropriate boundary conditions arising from  physical considerations on the model.

P. Macklin and J. Lowengrub~\cite{Macklin2007677} considered a ghost cell method for moving interface problems and applied it to a quasi--steady state reaction--diffusion model. However, the model is defined on a fixed domain, and the time--dependent interface is embedded in this fixed domain. The model we consider has an explicit moving boundary associated with it and hence the scheme in~\cite{Macklin2007677} does not directly apply. M. C. Calzada et al.~\cite{CarmenCalzada20111335} use a fictitious domain method to capture the time--dependent boundary. In a sense, we combine the synergy of both of these works: the time--dependent boundary problem is transformed to a fixed boundary problem without introducing any additional variables as in a level set method. Instead, we use an unknown variable in the model itself to characterise the moving boundary. The major contributions of this article are as follows: 
\begin{enumerate}[label= (\arabic*),ref= (\arabic*),leftmargin=\widthof{()}+3\labelsep]
	\item\label{it:model_1} A mathematically well--defined model that does not assume symmetric tumour growth is developed by adapting previous models.
	\item Two variants of this model depicting the tumour growth in (a) free suspension and (b) \emph{in vivo} surrounded by tissues or \emph{in vitro} in a passive polymeric gel are presented.
	\item We construct an extended model defined in a fixed domain and solutions of this model are used to recover solutions of the original model. Since no additional variables are introduced to achieve this (as in level set methods), the complexity of the model is not increased.
	\item We consider a numerical scheme based on finite volume methods, Lagrange $\mathbb{P}_2-\mathbb{P}_1$ Taylor--Hood finite element method, and mass--lumped finite element methods. The numerical scheme eliminates the need for re--meshing the time--dependent domain at each time step, which makes the computations economical.
	\item The numerical results are consistent with the findings from previous literature. We demonstrate the  versatility of the scheme in simulating initial tumour geometries with irregular and asymmetric shape and tumours with a changing topological structure. 
\end{enumerate}

The paper is organised as follows. In Section~\ref{sec:model_presentation}, we present the model assumptions, variables, and corresponding governing equations. The preliminaries and notations are presented in Section~\ref{sec:prelims}. In Section~\ref{sec:weak_equiv}, we present the notion of weak solutions and the main theorem that yields the equivalence between two different weak solutions in an appropriate sense. In Section~\ref{sec:n_scheme}, we provide the discretisation of the spatial and temporal domains and details of the numerical scheme. In Section~\ref{sec:n_results}, we apply the numerical scheme presented in Section~\ref{sec:n_scheme} to cases under different growth conditions and discuss the results in detail along with the scope for future research.

\section{Model presentation}
\label{sec:model_presentation}
The temporal and spatial variables are respectively denoted by $t$ and ${\boldsymbol x} := (x_{i})_{i=1,\ldots,d}$ $(d = 2\text{ or 3})$ in the sequel. All equations and parameters are presented in dimensionless form. In the case $d=2$, we take ${\boldsymbol x} = (x,y)$. At time $t \in (0,T)$, the tumour occupies the spatial domain $\Omega(t)$ in $\mathbb{R}^d$. The initial domain $\Omega(0)$ is a part of the given data. The tumour occupies the time--space domain  $D_T := \cup_{t \in (0,T)} (\{t\} \times \Omega(t))$. We assume that $\Omega(t)$ is a bounded domain with a $\mathscr{C}^{1}$--regular boundary~\cite[p.~627]{EvansPDE} given by $\Gamma(t) = \partial \Omega(t)$ for $t \in [0,T)$.  The time--dependent boundary $B_T := \stboundary$ of $D_T$ is also assumed to be $\mathscr{C}^1$--regular with respect to the time and space variables (see Figure~\ref{fig:2dgeometry}). Let $\Omega_\ell = (-\ell,\ell)^d$ be a domain in $\mathbb{R}^d$ such that $\Omega(t) \subset \Omega_\ell$ for every $t \in [0,T)$, which ensures $D_T \subset \mathscr{D}_T = (0,T) \times \Omega_\ell$. Let  ${\boldsymbol n}_{\vert{\Gamma(t)}}$ be the  unit normal to $\Gamma(t)$ pointing from $\Omega(t)$ and ${\boldsymbol n}_{\vert B_T}$ be the (time--space) unit normal to  $B_T$ pointing from $D_T$. If $\Omega(t) \subset \mathbb{R}^2$, then ${\boldsymbol \tau}_{\vert\Gamma(t)}$ denotes the unit tangent vector to $\Gamma(t)$. The projection of ${\boldsymbol u}$ on the tangent space of $\partial A$, where $A \subset \mathbb{R}^d$ is denoted by $\boldsymbol{u}_{\partial A,{\boldsymbol \tau}}$, which is defined by
$\boldsymbol{u}_{\partial A,{\boldsymbol \tau}} := ({\boldsymbol u}_{\vert \partial A}\cdot\boldsymbol{\tau}_{\vert \partial A}){\boldsymbol \tau}_{\vert \partial A}$ in two spatial dimensions and $\boldsymbol{u}_{\partial A,{\boldsymbol \tau}} := {\boldsymbol n}_{\vert \partial A} \times ({\boldsymbol u}_{\vert \partial A} \times {\boldsymbol n}_{\partial A})$ in three spatial dimensions.

\begin{figure}[h!] 
	\centering
	\includegraphics[scale=0.8]{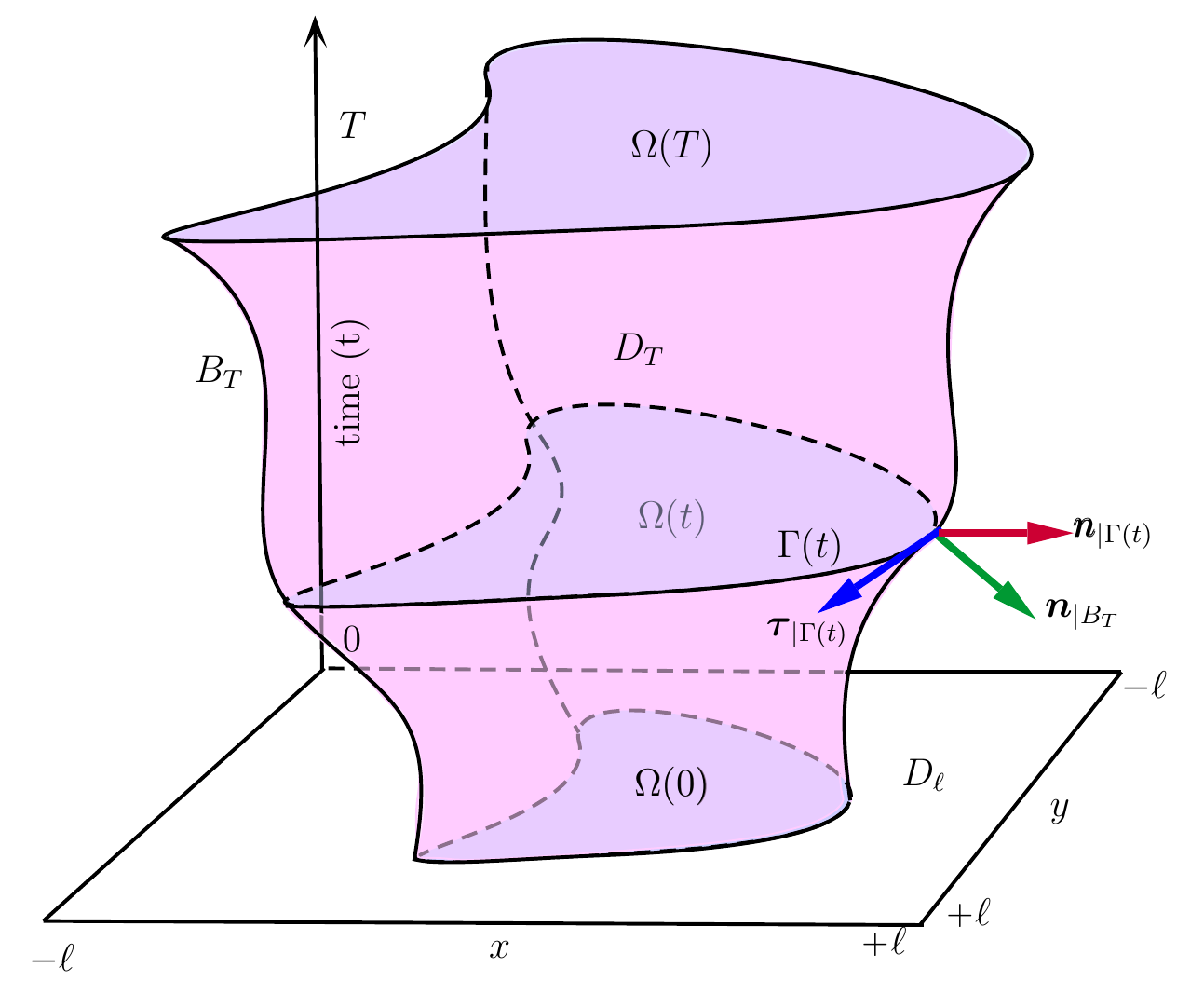}
	\caption{Three dimensional time--space domain occupied by the tumour. Here, $2\ell$ is the side length of the square $\Omega_\ell \subset \mathbb{R}^2$, $T$ is the final time of tumour growth, $\Omega(t) \subset \Omega_\ell$ is the domain occupied by the tumour at time $t$, $\Gamma(t)$ is the boundary of $\Omega(t)$, $D_T$ is the time--space domain $\cup_{0 < t < T} (\{t\} \times \Omega(t))$, $B_T$ (pink envelope) is the evolving boundary given by $\stboundary$, and $\mathscr{D}_T$ is the time--space domain $(0,T) \times \Omega_\ell$.}
	\label{fig:2dgeometry}
\end{figure}

The relative volume of tumour cells (cell phase) and extra--cellular fluid (fluid phase) are denoted by $\alpha := \alpha(t,{\boldsymbol x})$ and $\beta := \beta(t,{\boldsymbol x})$, respectively. We assume that the tumour does not contain any voids, which implies that $\alpha + \beta = 1$, and hence $\beta$ can be determined using $\alpha$. The velocity by which the cells are moving is denoted by ${\boldsymbol u} := {\boldsymbol u}(t,{\boldsymbol x})$. The average pressure experienced in the fluid phase is denoted by $p := p(t,{\boldsymbol x})$. The cell growth is controlled by a limiting nutrient and $c := c(t,{\boldsymbol x})$ represents its concentration. 

Depending on the conditions in which the tumour is growing, the nutrient supply can be abundant or limited.  For instance, when the growth is \emph{in vitro}, the external atmosphere acts as an unlimited source of nutrients, like oxygen. On the contrary, when the growth is \emph{in vivo}, the tissues and other biological materials around the tumour hinder the smooth diffusion of nutrients from the adjacent capillary tissues. Hence, the nutrient supply is limited in the \emph{in vivo} case.  We consider the two cases of \emph{in vitro} and \emph{in vivo} growth, and present models to describe them.  

\subsection{Common features of both models}

The \emph{in vitro} model comes from \cite{byrne_mcelawin_preziosi_2003}, and the \emph{in vivo} one is a slight modification of this model. Both models are presented in dimensionless form and seek the variables $\left(\alpha,\velc,p,c,\Omega\right)$ such that the mass balance on $\alpha$ and the momentum balance on $(\boldsymbol{u},p)$ hold in the moving domain: for every $t \in (0,T)$ and ${\boldsymbol x} \in \Omega(t)$,
\begin{subequations} \label{sys:model_common}
	\begin{align}
	\label{eqn:v_frac}
	\partial_t \alpha + \text{div}(\alpha\velc) &= \alpha f(\alpha,c), \\
	\label{eqn:velocity}
	-\text{div} \left( \alpha \varepsilon(\velc)\right) + \nabla p &= -\nabla\mathscr{H}(\alpha),  \textrm{ and }\\
	\label{eqn:pressure}
	- \text{div} \left(\dfrac{1-\alpha}{k\alpha} \nabla p \right) + \text{div} (\velc) &= 0.
	\end{align}
	\label{eqn:formulation I}
	The difference between the two models lies in the domain over which the oxygen tension satisfies the following reaction--diffusion equation:
	\begin{equation}
	\label{eqn:o_tension}
	\partial_t c - \,\text{div}(\eta \nabla c) =  -\dfrac{Q c \alpha}{1+\widehat{Q}c}. 
	\end{equation}
	Above, the function $f$ is defined by $f(\alpha,c) := (1 - \alpha)b(c) -  d(c)$, where $b(c) := (1 + s_1)c/(1 + s_1c)$, $d(c) := (s_2 + s_3c)/(1 + s_4c)$, and $s_1,\,s_2,\,s_3$ and $s_4$ are positive constants which control proliferation and death rates of the tumour cells. The operator $\varepsilon$ is defined by $\varepsilon({\boldsymbol u}) := 2\mu\nabla_{s}{\boldsymbol u}+ \lambda \text{div}({\boldsymbol u})\mathbb{I}_{d}$, where $\mathbb{I}_d$ is the $d$--dimensional identity tensor and $\nabla_{s}{\boldsymbol u} =  (\nabla {\boldsymbol u} + (\nabla {\boldsymbol u})^T)/2$. The scalar constants $\mu$ and $\lambda$ are the shear and bulk viscosity coefficients, respectively and are related by $\lambda = -2\mu/3$ and $\mu > 0$. The function $\mathscr{H}(\alpha)$ is defined by $\alpha(\alpha - \astar)^{+}/(1 - \alpha)^2$, where $\astar$ is a positive constant, $s^{+} := \max(0,s)$, and $s^{-} := -\min(0,s)$ in the sequel. The positive constant $k$ controls the traction between the cell and fluid phases. The constant $\eta > 0$ is the diffusivity coefficient of the limiting nutrient inside the tumour, and the constants $Q > 0$, further referred to as the absorptivity coefficient, and $\widehat{Q} \ge 0$ control the nutrient consumption by the cells. 
	
	The initial condition on $\alpha$ and the boundary conditions on $(\boldsymbol{u},p)$ are also common to both models:
	\begin{equation}
	\alpha(0,{\boldsymbol x}) = \alpha_0({\boldsymbol x})\quad\forall{\boldsymbol x} \in \Omega(0),
	\label{eqn:in_cond} 
	\end{equation}
	\begin{equation}
	(- \alpha \varepsilon({\boldsymbol u}) + p\mathbb{I}_d)\cdot {\boldsymbol n}_{\vert {\Gamma(t)}} = -\mathscr{H}(\alpha)\mathbb{I}_d\cdot{\boldsymbol n}_{\vert {\Gamma(t)}},\;\boldsymbol{u}_{\Gamma(t),{\boldsymbol \tau}} = \boldsymbol{0}\,,\;
	p_{\vert{\Gamma(t)}} = 0\quad\forall t \in (0,T).
	\label{eqn:bound_cond_2}
	\end{equation}
	The moving boundary is governed by the ordinary differential equation:
	\begin{equation}
	\partial_t {\boldsymbol \gamma}\cdot{\boldsymbol n}_{\vert{\Gamma(t)}} = 	{\boldsymbol u}_{\vert{\Gamma(t)}}\cdot{\boldsymbol n}_{\vert{\Gamma(t)}}\quad\forall t\in (0,T),
	\label{bdrequation} 
	\end{equation}	
\end{subequations}
where ${\boldsymbol \gamma}$ is a local parametrisation of $B_T$. We assume that $0 < m_{01} \leq  \alpha_0({\boldsymbol x}) \leq m_{02} < 1$ and $0 \leq c_0({\boldsymbol x}) \leq 1$ for every ${\boldsymbol x} \in \Omega(0)$, where $m_{01}$ and $m_{02}$ are positive constants.
\begin{remark}
	Note that in~\eqref{bdrequation} we only specify the normal velocity of the moving boundary. The tangential velocity is not provided here. This is because tangential velocity does not change the topological structure of $B_T$, but changes only the parametrisation of $B_T$. Therefore, the domain $D_T$, that is the time--space region enclosed by $B_T$, is independent of the tangential velocity of the moving boundary. The extended solution presented in Definition \ref{defn:nsm_ext} below recovers the domain $D_T$ without resorting to an explicit parametrisation of the boundary $B_T$, and is an added advantage of the notion of the extended solution. 
\end{remark}

The initial and boundary conditions for $c$ depend on each model and are made precise in the next sections. Table \ref{tab:NLM.NUM} summarises the two models.

\begin{table}
	\begin{center}
		\begin{tabular}{|c|c|c|}
			\hline
			\emph{Equation} & \emph{NUM} & \emph{NLM}\\
			\hline
			Evolution of $\alpha, \boldsymbol{u}, p$, Eqs. \eqref{eqn:v_frac}--\eqref{eqn:pressure} & \multicolumn{2}{c|}{$\boldsymbol{x}\in \Omega(t)$}\\
			\hline
			Boundary conditions $\alpha, \boldsymbol{u}, p$ & \multicolumn{2}{c|}{\eqref{eqn:bound_cond_2}} \\
			\hline
			Initial conditions on $\alpha$ & \multicolumn{2}{c|}{\eqref{eqn:in_cond}}  \\
			\hline
			Evolution of $c$, Eq. \eqref{eqn:o_tension}  & For $\boldsymbol{x}\in \Omega(t)$ & For $\boldsymbol{x}\in\Omega_\ell$\\
			\hline
			Initial conditions on $c$ & \eqref{eqn:in_cond_NUM_c} (on $\Omega(0)$) & \eqref{eqn:in_cond_NLM_c} (on $\Omega_\ell$)  \\
			\hline
			Boundary conditions $c$ & \eqref{eqn:bound_cond_NUM_c} (on $\Gamma(t)$) & \eqref{eqn:bound_cond_NLM_c} (on $\partial\Omega_\ell$)\\
			\hline
		\end{tabular}
		\caption{Summary of NUM and NLM models.}
		\label{tab:NLM.NUM} 
	\end{center}
\end{table}

\subsection{Nutrient unlimited model (NUM)}
In the \emph{nutrient unlimited model} (NUM), we assume that the tumour grows in free space. Since the tumour has no voids within and is close--packed, it is reasonable to assume that the nutrient diffusion rate in the tumour is much lower than that of the free space outside the tumour. The nutrient consumed by the boundary cells is immediately replenished by the fast diffusing external nutrient supply. As a consequence, the oxygen tension equation \eqref{eqn:o_tension} is only solved on the moving domain, for $t\in (0,T)$ and $\boldsymbol{x}\in\Omega(t)$, and at the boundary of this moving domain the nutrient concentration is set as the maximum value, which is unity after non--dimensionalisation. This leads to the following boundary and initial conditions for $c$:
\begin{align}
c_{\vert\Gamma(t)} ={}& 1\quad\forall t \in (0,T),
\label{eqn:bound_cond_NUM_c}\\
c(0,{\boldsymbol x}) ={}& c_0({\boldsymbol x})\quad\forall{\boldsymbol x} \in \Omega(0),
\label{eqn:in_cond_NUM_c} 
\end{align}

\subsection{Nutrient limited model (NLM)}

In the \emph{nutrient limited model} (NLM), we assume that the tumour is growing inside a medium or a tissue. In this case, the nutrient diffusion rates in the exterior and interior regions of the tumour are in the same numerical range. Therefore, considerable delay can be expected for the nutrient to diffuse through the medium and reach the tumour. Consequently, the nutrient concentration at the tumour boundary is not unity at every time and one has to model the diffusion of the nutrient in the medium and in the tumour. Taking $\Omega_\ell$ as the spatial region that encloses the tumour and the medium, the oxygen tension equation \eqref{eqn:o_tension} is therefore solved for $t\in (0,T)$ and $\boldsymbol{x}\in\Omega_\ell$
($\eta$ could change between the external medium and the tumour), and the boundary and initial conditions on $c$ are
\begin{align}
c(t,\boldsymbol{x}) ={}& c_b({\boldsymbol x})\quad\forall t \in (0,T)\,,\forall \boldsymbol{x}\in \partial\Omega_\ell,
\label{eqn:bound_cond_NLM_c}\\
c(0,{\boldsymbol x}) ={}& 0\quad\forall{\boldsymbol x} \in \Omega_\ell,
\label{eqn:in_cond_NLM_c} 
\end{align}
This initial condition means that no nutrient is available for the tumour cells initially.
The boundary data satisfy $0 \leq c_{b} \leq 1$, and depends on the modelling situation under consideration. For illustrative purposes in two dimensions, we assume that blood vessels are present at $y = -\ell$ or $x = -\ell$ only. Therefore, the  nutrient concentration at the boundary, $c_b$, is unity at $y = -\ell$ or $x = -\ell$ and zero at the other points in $\partial \Omega_\ell$.

\section{Preliminaries and notations}
\label{sec:prelims}
We describe a smooth hypersurface, $\mathcal{S} \subset \mathbb{R}^d$ and a local parametrisation of $\mathcal{S}$. For a detailed discussion on these topics, the reader may refer to~\cite[Chapter 2]{KTapp2016}.  The notion of the local parametrisation of a smooth surface is crucial in extending the NUM and NLM models defined in $D_T$ to $\mathscr{D}_T$, and thereby in eliminating the need for the evolving boundary, $B_T$.
\begin{definition}[{$\mathscr{C}^1-$}smooth hypersurface]
	\label{defn:local_rep}
	A set $\mathcal{S} \subset \mathbb{R}^d$ is said to be a $\mathscr{C}^1-$smooth hypersurface in $\mathbb{R}^d$ if the following conditions hold:
	\begin{enumerate}[label= $\mathrm{(SH.\arabic*)}$,ref=$\mathrm{(SH.\arabic*)}$,leftmargin=\widthof{(SH.1)}+3\labelsep]
		\item \label{it.sh1} For each ${\boldsymbol z} \in \mathcal{S}$, there exists an open set $\mathcal{O}_{\boldsymbol z} \subset \mathbb{R}^d$ containing ${\boldsymbol z}$ and a function $f_{\boldsymbol z} : \mathcal{O}_{\boldsymbol z} \rightarrow \mathbb{R}$ such that $\mathcal{S}\cap\mathcal{O}_{\boldsymbol z} = \{ {\boldsymbol x} \in \mathcal{O}_{\boldsymbol z} : f_{\boldsymbol z}(\boldsymbol x) = 0 \}$.
		\item Each $f_{\boldsymbol z}$ in~\ref{it.sh1} belongs to $\mathscr{C}^1(\mathcal{O}_{\boldsymbol z})$ and  $\nabla f_{\boldsymbol z} \neq 0$ on $\mathcal{O}_{\boldsymbol z}$.
	\end{enumerate}
	The collection $\{\mathcal{O}_{\boldsymbol z}, f_{\boldsymbol z}\}_{{\boldsymbol z} \in \mathcal{S}}$ is called a $\mathscr{C}^1-$smooth local representation of $\mathcal{S}$.
\end{definition}

\begin{definition}[Regular surface and local parametrisation] \label{defn:local_param}
	A set $\mathcal{S} \subset \mathbb{R}^d$ is said to be a regular surface if for each ${\boldsymbol z} \in \mathcal{S}$, there exists open sets $U_{\boldsymbol z} \subset \mathbb{R}^{d-1}$ and $V_{\boldsymbol z} \subset \mathbb{R}^d$ with ${\boldsymbol z} \in V_{\boldsymbol z}$, and a diffeomorphism ${\boldsymbol \sigma}_{\boldsymbol z} : U_{\boldsymbol z} \rightarrow V_{\boldsymbol z}\cap \mathcal{S}$. Each ${\boldsymbol \sigma}_{\boldsymbol z}$ is called a coordinate chart, and the collection $\{U_{\boldsymbol z}, V_{\boldsymbol z}, {\boldsymbol \sigma}_{\boldsymbol z}\}_{{\boldsymbol z} \in \mathcal{S}}$ is called a local parametrisation for $\mathcal{S}$.
\end{definition} 

\noindent  If $\{\mathcal{O}_{\boldsymbol z}, f_{\boldsymbol z}\}_{{\boldsymbol z} \in \mathcal{S}}$ is a $\mathscr{C}^1-$smooth local representation of the $\mathscr{C}^1 -$smooth hypersurface $\mathcal{S}$, then the normal to $\mathcal{S}$ at a point ${\boldsymbol z} \in \mathcal{S}$ is given by $\nabla f_{\boldsymbol z}({\boldsymbol z})/\vert\vert\nabla f_{\boldsymbol z}({\boldsymbol z})\vert\vert_{2}$, and this is meaningful since $\nabla f_{\boldsymbol z}({\boldsymbol z}) \neq 0$ by Definition~\ref{defn:local_rep}. An application of Theorem 3.27 in~\cite{KTapp2016} shows that every $\mathscr{C}^1-$smooth hypersurface is regular and therefore, has a local parametrisation.

\subsection{Function spaces and norms}
In this subsection, we give the definitions of function spaces and norms used in the remaining of this article.

For a domain $A \subset \mathbb{R}^d$, $L^{p}(A)$ ($1 \leq p \leq \infty$) and $H^1(A)$ are standard Sobolev spaces of functions $f : A \rightarrow \mathbb{R}$. The notation $(\cdot,\cdot)_{A}$ stands for the standard $L^2(A)$ inner product. The space $\hna = (H^1(A))^d$ is the collection of functions ${\boldsymbol u} = \left({u}_1,\ldots,{u}_d\right)$ such that ${u}_i : A \rightarrow \mathbb{R}$  and ${u}_i \in H^1(A)$ for $i=1,\dots,d$. 

We define the norms $||{\boldsymbol u}||_{0,A} := ({\boldsymbol u},{\boldsymbol u})_{A}^{1/2}$ and $||{\boldsymbol u}||_{k,A} := \sum_{i=1}^d\sum_{\boldsymbol{j}, |\boldsymbol{j}| \leq k} ||\partial^{\boldsymbol{j}} u_i||_{0,A}$, where $\boldsymbol{j}$ is a multi-index. Define the subspace of functions in $\boldsymbol{\mathrm H}^1_{d}(A)$ with homogeneous tangential component at $\partial A$, and the subspace of functions in $H^1(A)$ with homogeneous Dirichlet boundary condition $\partial A$, respectively by
\begin{align*}
\utanspacea &:= \{ {\boldsymbol u} \in \boldsymbol{\mathrm H}^1_{d}(A) : \boldsymbol{u}_{\partial A,{\boldsymbol \tau}} = {\boldsymbol 0}\} \textrm{ and } \\
H^1_0(A) &:= \{ f \in H^1(A) : f_{\vert \partial A} = 0\}.
\end{align*}
The space $BV(A)$ denotes the the space of all functions with bounded variation (see Definition~\ref{def:bv}) on the set $A$.

Let $A_T = \cup_{0 < t < T} \{t\} \times X(t)$, where $\{X(t)\}_{t \in (0,T)}$ is a family of domains such that $X(t) \subset \mathbb{R}^d$ for every $t \in (0,T)$. Define the Hilbert spaces
\begin{align}
H_{\nabla}^{1,u}(A_T) := &\{ \velc \in (L^2(A_T))^d : \partial_{x_j}u_i \in L^2(A_T),\,i,j=1,\ldots,d \\
&{}{}{}{}\text{ and } \boldsymbol{u}_{\partial X(t),\boldsymbol{\tau}} = \boldsymbol{0}\;\forall t \in (0,T)\} \textrm{ and }\\
H_{\nabla}^{1,c}(A_T) := &\{c \in L^2(A_T) : \nabla c \in (L^2(A_T))^d \text{ and } c_{\vert_{\partial X(t)}} = 0\;\forall t \in (0,T) \}.
\end{align}


\section{Weak solutions and equivalence theorem}
\label{sec:weak_equiv}

In this section, we first establish in Section \ref{sec:wp.up} the well--posedness of the weak form of the velocity--pressure momentum balance, and present two weak formulations of the NUM model \eqref{sys:model_common}--\eqref{eqn:in_cond_NUM_c}. In the first one, the scalar conservation law \eqref{eqn:v_frac} is set on the moving domain $\Omega(t)$, while in the second one the velocity and oxygen tension are extended to the entire box $\Omega_\ell$ and the cell volume fraction $\alpha$ is set to satisfy the conservation law \eqref{eqn:v_frac} on this box. The interest of this second model, as already illustrated in the one dimensional case in \cite{DNR19,Remesan2019}, is to enable the usage of a discrete scheme using a fixed background mesh, rather than a mesh that moves with the domain $\Omega(t)$. 

The two weak formulations are shown in Section \ref{sec:wf.equiv} to be equivalent. The key relation for establishing this equivalence is Proposition \ref{prop:normal}, which establishes a formula for the outer normal to the time--space tumour domain in terms of the cell volume fraction, as well as the fact that if a piecewise smooth vector field $\mathbf{F}$ has an $L^2$ divergence, then it has a zero normal jump across any hypersurface. 

We only consider here the NUM model, the extension to NLM being straightforward.

\subsection{Well-posedness of velocity-pressure system}\label{sec:wp.up}
We present the weak formulations of~\eqref{eqn:velocity} and~\eqref{eqn:pressure} with boundary conditions \eqref{eqn:bound_cond_2}, which remain the same for Definition~\ref{defn:nsm_weak} and Definition~\ref{defn:nsm_ext}. Let ${\boldsymbol u} \in \uspace$ and $p\in \cspace$. The weak formulations are as follows. For all ${\boldsymbol v} \in \uspace$ and $z\in \cspace$, and for each $t \in (0,T)$ it holds
\begin{subequations}
	\begin{align}
	\label{eqn:pvsys_1}
	a_{1}^t(\velc(t,\cdot),{\boldsymbol v}(t,\cdot)) - a_3^t(p(t,\cdot),{\boldsymbol v}(t,\cdot)) &=   \mathcal{L}_{\alpha}^t({\boldsymbol v}(t,\cdot)) \textrm{ and } \\
	\label{eqn:pvsys_2}
	a_{2}^t(p(t,\cdot),z(t,\cdot)) + a_3^t(z(t,\cdot),\velc(t,\cdot)) &= 0,
	\end{align}
\end{subequations}
where $a_{1}^t : \utanspace \times \utanspace \rightarrow \mathbb{R}$, $a_2^t :  H^1_0(\Omega(t)) \times H^1_0(\Omega(t)) \rightarrow \mathbb{R},$ and $a_{3}^t : H^1_0(\Omega(t)) \times \utanspace \rightarrow \mathbb{R}$ are bilinear forms given by: for ${\boldsymbol \psi}_j \in \utanspace$ and $q_j \in H^1_0(\Omega(t))$, where $j \in \{1,2\}$,
\begin{align}
a_{1}^t({\boldsymbol \psi}_1,{\boldsymbol \psi}_2) &= \int_
{\Omega(t)} \alpha(t,\cdot)\left(2\mu \nabla_s{\boldsymbol \psi}_1:\nabla_s{\boldsymbol \psi}_2 + \lambda \text{div}({\boldsymbol \psi}_1)\text{div}({\boldsymbol \psi}_2)\right)\vx, \\
a_{2}^t(q_1,q_2) &= \int_{\Omega(t)} \dfrac{1-\alpha(t,\cdot)}{k\alpha(t,\cdot)} \nabla q_1 \cdot \nabla q_2\,\mathrm{d}{\boldsymbol x}, \textrm{ and } \\
a^t_3(q_1,{\boldsymbol \psi}_1) &= \int_{\Omega(t)}  q_1\,\text{div}({\boldsymbol \psi}_1)\,\vx,
\end{align}
and $\mathcal{L}_{\alpha}^t : \hn \rightarrow \mathbb{R}$ is a linear form given by
\begin{equation} \label{eqn:lform}
\mathcal{L}_{\alpha}^t({\boldsymbol \psi}_1) = \int_{\Omega(t)} \mathscr{H}(\alpha(t,\cdot))\text{div}({\boldsymbol \psi}_1)\,\vx.
\end{equation}
Under the assumption that $\alpha : D_T \rightarrow \mathbb{R}$ is known and satisfies $0 < m_{11} \le \alpha \le m_{12} < 1$, where $m_{11}$ and $m_{12}$ are positive constants, we show that for each $t \in (0,T)$, ~\eqref{eqn:pvsys_1} and~\eqref{eqn:pvsys_2} are well-posed. In Theorem~\ref{thm:well_posedness_up}, we suppress the time dependency for the ease of notation; hence, ${\boldsymbol u}$ in Theorem~\ref{thm:well_posedness_up} stands for ${\boldsymbol u}(t,\cdot)$, and so do ${\boldsymbol v}, p,$ and $z$.

\begin{lemma}
	\label{lemma:coer_eps}
	If $\boldsymbol{v} \in \utanspace$, then there exists a constant $\mathscr{C}_{KP} > 0$ such that $\mathscr{C}_{KP} ||{\boldsymbol v}||_{1,\Omega(t)} \le ||\nabla_s({\boldsymbol v})||_{0, \Omega(t)}$.
\end{lemma}
\begin{proof}
	Consider the spaces $X = \utanspace$, $Y = [L^2(\Omega(t))]^{d \times d}$, and $Z = [L^2(\Omega(t))]^{d}$, and the linear map $A := \nabla_s : X \rightarrow Y$ and the natural embedding $T := id: X \rightarrow Z$.  Theorem~13 in \cite{Bauer2016173} shows that $A$ is an injection. The natural embedding $T$ is compact by Rellich-Kondrachov Theorem. Korn's second inequality (Theorem~\ref{korns}) yields 
	$\mathscr{C}_K ||\boldsymbol{v}||_{1,\Omega} = \mathscr{C}_K ||\boldsymbol{v}||_{X} \le ||\nabla_s (\boldsymbol{v})||_{0,\Omega} + ||\boldsymbol{v}||_{0,\Omega} = ||A \boldsymbol{v}||_Y + ||T\boldsymbol{v}||_Z$. An application of Petree--Tartar lemma (Theorem~\ref{tartar}) yields the desired conclusion. 
\end{proof}

\begin{theorem}[Well-posedness]
	\label{thm:well_posedness_up}
	Define the product space $\upspace{t} := \utanspace \times H^1_0(\Omega(t))$ and the bilinear operator $\mathfrak{A}^t : \upspace{t} \times \upspace{t} \rightarrow \mathbb{R}$ by
	\begin{equation}
	\mathfrak{A}^t\left(({\boldsymbol u},p),\,({\boldsymbol v},z) \right) = a_{1}^t({\boldsymbol u},{\boldsymbol v}) - a_{3}^t(p,{\boldsymbol v}) + a_{2}^t(p,z) + a_{3}^t(z,{\boldsymbol u}).
	\end{equation}
	If $0 < m_{11} \le \alpha \le m_{12} < 1$, then $\mathfrak{A}^t$ is a continuous and coercive bilinear form in $\upspace{t}$, and the linear form $\mathfrak{L}^t : \upspace{t} \rightarrow \mathbb{R}$ defined by $\mathfrak{L}^t({\boldsymbol v},z) = \mathcal{L}_{\alpha}^t({\boldsymbol v})$ is continuous on $\upspace{t}$. Hence,  there exists a unique $({\boldsymbol u},p) \in \upspace{t}$ such that for all $({\boldsymbol v},z) \in \upspace{t}$,
	\begin{equation}
	\mathfrak{A}^t\left(({\boldsymbol u},p),\,({\boldsymbol v},z) \right) = \mathfrak{L}^t(({\boldsymbol v},z)).
	\label{eqn:existence_up}
	\end{equation}
\end{theorem}
\begin{proof}
	Continuity of the bilinear form follows from the estimates below. Since $||\textrm{div}({\boldsymbol u})||_{0,\Omega(t)} \leq \sqrt{d} ||{\boldsymbol u}||_{1,\Omega(t)}$,
	\begin{align}
	\mathfrak{A}^t\left(({\boldsymbol u},p),\,({\boldsymbol v},z) \right) &\le 2m_{12}(\mu + \lambda) ||{\boldsymbol u}||_{1,\Omega(t)}||{\boldsymbol v}||_{1,\Omega(t)} + ||p||_{1,\Omega(t)}\sqrt{d}||\boldsymbol{v}||_{1,\Omega(t)} \\
	&{}{}+ \dfrac{1 - m_{11}}{km_{11}} ||p||_{1,\Omega(t)}||z||_{1,\Omega(t)} + \sqrt{d}||z||_{1,\Omega(t)}||{\boldsymbol u}||_{1,\Omega(t)} \\
	&\le\mathscr{C} (||{\boldsymbol u}||_{1,\Omega(t)}^2 + ||p||_{1,\Omega(t)}^{2})^{1/2}(||{\boldsymbol v}||_{1,\Omega(t)}^2 + ||z||_{1,\Omega(t)}^{2})^{1/2},
	\end{align}
	where $\mathscr{C}$ is a constant. Set ${\boldsymbol v} = {\boldsymbol u}$ and $z = p$ in $\mathfrak{A}^t\left(({\boldsymbol u},p),\,({\boldsymbol v},z) \right)$ to obtain,
	\begin{align}
	\mathfrak{A}^t\left(({\boldsymbol u},p),\,({\boldsymbol u},p) \right) &= a_1^t({\boldsymbol u},{\boldsymbol u}) + a_2^t(p,p) \nonumber \\
	&\ge 2m_{11}\mu \int_{\Omega(t)} \nabla_s{\boldsymbol u}:\nabla_s{\boldsymbol u}\,\vx  + \dfrac{1 - m_{12}}{km_{12}} ||p||_{1,\Omega(t)}^2.
	\end{align}
	Then, Lemma~\ref{lemma:coer_eps} yields the coercivity of $\mathfrak{A}^t$. The following estimate yields the continuity of $\mathfrak{L}^t$:
	\begin{align}
	\mathfrak{L}^t({\boldsymbol v},z) &\le \sqrt{2}\max(1,\mathscr{H}(m_{12})\sqrt{d\mu_{\mathbb{R}^d}(\Omega(t))})(||{\boldsymbol v}||_{1,\Omega(t)}^2 + ||z||_{1,\Omega(t)}^2)^{1/2},
	\end{align}
	where $\mu_{\mathbb{R}^d}$ is the $d$-dimensional Lebesgue measure. An application of Lax-Milgram theorem establishes the existence of a unique $({\boldsymbol u},p) \in \upspace{t}$ such that~\eqref{eqn:existence_up} (hence,~\eqref{eqn:pvsys_1} and~\eqref{eqn:pvsys_2}) holds.
\end{proof}

\begin{definition}[NUM--weak solution] \label{defn:nsm_weak}
	A weak solution of the NUM in $D_T$, further referred to as NUM--weak solution,  is a five-tuple $(\alpha,\velc,p,c,\Omega)$ such that~\ref{it.sw1}-\ref{it.sw4} hold.
	\begin{enumerate}[label= $\mathrm{(SW.\arabic*)}$,ref=$\mathrm{(SW.\arabic*)}$,leftmargin=\widthof{(SW.1)}+3\labelsep]
		\item\label{it.sw1} The volume fraction satisfies $\alpha  \in L^\infty(D_T)$, $0 < m_{11} \le \alpha \le m_{12} < 1$, where $m_{11} \le m_{01}$ and $m_{02} \le m_{12}$ are  constants, and $\forall\,\varphi \in \mathscr{C}_c^{\infty}(\overline{D}_T\backslash (\{T\} \times \Omega(T)))$
		\begin{multline}
		\int_{D_T} (\alpha,\,\alpha\velc)\cdot \nabla_{(t,{\boldsymbol x})}\varphi\,\mathrm{d}t\vx + \int_{\Omega(0)} \varphi(0,{\boldsymbol x}) \alpha_0({\boldsymbol x})\,\vx + \int_{D_T} \alpha f(\alpha,c) \varphi\,\mathrm{d}t\vx, 
		\\
		\label{hyper_wf}
		= \int_{B_T}(\alpha, \velc \alpha)\cdot {\boldsymbol n}_{\vert B_T} \varphi\,\mathrm{d}s.
		\end{multline}
		\item\label{it.sw2} The velocity ${\boldsymbol u} \in \uspace$ and pressure $p \in \cspace$ satisfy~\eqref{eqn:pvsys_1} and~\eqref{eqn:pvsys_2} for every ${\boldsymbol v} \in \uspace$ and $z \in \cspace$.
		\item\label{it.sw3} The nutrient concentration is such that $c - 1 \in \cspace$, $c \ge 0$, and $\forall\, \zeta \in \cspace$ with $\partial_t\zeta \in L^2(D_T)$
		\begin{multline}
		-\int_{D_T} c\,\partial_t \zeta\,\vx\,\mathrm{d}t -  \int_{D_T} \eta \nabla c  \cdot \nabla \zeta\,\vx\,\mathrm{d}t  + \int_{\Omega(0)} c_0({\boldsymbol x})\zeta(0,{\boldsymbol x})\,\vx \\
		+ \int_{D_T}\dfrac{Q c \alpha}{1+\widehat{Q}_1c} \zeta\,\vx\,\mathrm{d}t  = 0. 
		\label{eqn:ot_wnsm}
		\end{multline}
		\item\label{it.sw4} The time-dependent boundary $\Gamma(t)$ is governed by~\eqref{bdrequation}.
	\end{enumerate}
\end{definition}

\begin{definition}[NUM--extended solution] \label{defn:nsm_ext}
	A weak solution of the NUM in $\mathscr{D}_T$, further referred to as NUM--extended solution, is a four-tuple $(\widetilde{\alpha},\widetilde{\boldsymbol u},\widetilde{p},\widetilde{c})$ such that \ref{it.se1}--\ref{it.se4} hold.
	\begin{enumerate}[label= $\mathrm{(SE.\arabic*)}$,ref=$\mathrm{(SE.\arabic*)}$,leftmargin=\widthof{(SE.1)}+3\labelsep]
		\item\label{it.se1} The function $\widetilde{\alpha}$ is such that $\widetilde{\alpha} \in L^{\infty}(\mathscr{D}_T)$, $\widetilde{\alpha} \ge 0$, and $\forall\,\widetilde{\varphi} \in \mathscr{C}_c^{\infty}([0,T) \times \Omega_\ell)$:
		\begin{multline}
		\int_{\mathscr{D}_T} (\alex,\velcex \alex)\cdot \nabla_{(t,{\boldsymbol x})}\widetilde{\varphi}\,\mathrm{d}t\,\vx + \int_{\Omega(0)} \widetilde{\varphi}(0,{\boldsymbol x}) \alpha_0({\boldsymbol x})\,\vx
		+ \int_{\mathscr{D}_T} \alex f(\alex,\Cex) \widetilde{\varphi}\,\mathrm{d}t\,\vx = 0.
		\label{eqn:huper_ex_wf}
		\end{multline}
		\item\label{it.se2} For a fixed $t$, define  $\widetilde{\Omega}(t) := \{ (t,{\boldsymbol x}) : \widetilde{\alpha}(t,{\boldsymbol x}) > 0 \}$ and $\widetilde{D}_T := \cup_{0 < t < T}\{t\} \times \widetilde{\Omega}(t)$. Then, it holds $\widetilde{\velc}_{\vert \mathscr{D}_T\backslash\overline{\widetilde{D}_{T}}} = {\boldsymbol 0}$, $\widetilde{p}_{\vert \mathscr{D}_T\backslash\overline{\widetilde{D}_{T}}} = 0$, and $\widetilde{c}_{\vert \mathscr{D}_T\backslash\overline{\widetilde{D}_{T}}} = 1$.
		\item\label{it.se3} The functions $\widetilde{\velc}_{\vert \widetilde{D}_{T}}$ and $\widetilde{p}_{\vert \widetilde{D}_{T}}$  is such that $\widetilde{\velc}_{\vert \widetilde{D}_{T}} \in H_{\nabla}^{1,u}(\widetilde{D}_T)$, $\widetilde{p}_{\vert \widetilde{D}_{T}} \in H_{\nabla}^{1,c}(\widetilde{D}_T)$ and satisfy  ~\eqref{eqn:pvsys_1}--\eqref{eqn:pvsys_2} with $\Omega(t)$, $D_T$, and $\alpha$ set as $\widetilde{\Omega}(t)$, $\widetilde{D}_T$, and $\widetilde{\alpha}_{\vert \widetilde{\Omega}(t)}$, respectively.
		\item\label{it.se4} The function $\widetilde{c}_{\vert{\widetilde{D}_T}}$ is such that $\widetilde{c}_{\vert{\widetilde{D}_T}} - 1 \in H_{\nabla}^{1,c}(\widetilde{D}_T)$ and satisfies ~\eqref{eqn:ot_wnsm} with $D_T$ set as $\widetilde{D}_T$ for all $\zeta \in H_{\nabla}^{1,c}(\widetilde{D}_T)$ with $\partial_t\zeta \in L^2(\widetilde{D}_T)$.
	\end{enumerate}
\end{definition}

\subsection{Equivalence of weak solutions}\label{sec:wf.equiv}
In this subsection, we show that Definitions~\ref{defn:nsm_weak} and~\ref{defn:nsm_ext} are equivalent to each other in an appropriate sense and under some regularity assumptions on $B_T$. In particular, we show that the recovered domain $\widetilde{D}_T$ in Definition~\ref{defn:nsm_ext} is equal to $D_T$ in Definition~\ref{defn:nsm_weak}.
\begin{definition}[Time projection map]
	The time projection map $\pi_t : \mathbb{R}^{+} \times \mathbb{R}^{d-1} \rightarrow \mathbb{R}^{+} \times \mathbb{R}^d$ is defined by $\pi_t(t,{\boldsymbol y}) = t$ for all $(t,{\boldsymbol y}) \in \mathbb{R}^{+} \times \mathbb{R}^{d-1}$.
\end{definition}

\begin{remark}[Time-slice property of $B_T$]
	\label{rem:time_slice}
	While constructing a local parametrisation for $B_T$ in the sense of Definition~\ref{defn:local_param}, we use time also as a parameter through the time projection map $\pi_t$ to preserve the `time-slice' geometry of $B_T = \cup_{t} \{t\} \times \partial\Omega(t)$ in the following way. Let $(\mathbb{R}^{+} \times U_{\boldsymbol \omega},\mathbb{R}^{+} \times V_{\boldsymbol \omega}, \sigma_{\boldsymbol w} = (\pi_t,{\boldsymbol \gamma}_{\boldsymbol \omega}))$  be a local parametrisation around ${\boldsymbol w} \in B_T$ of the  evolving boundary $B_T$ in the sense of Definition~\ref{defn:local_param}. Then, for a fixed time, $t$, the restriction $\{U_{\boldsymbol \omega},V_{\boldsymbol \omega},{\boldsymbol \gamma}_{\boldsymbol \omega}(t,\cdot)\}_{\boldsymbol \omega \in \{t\} \times \partial \Omega(t)}$ is a local parametrisation of $\partial \Omega(t)$. The time-slice structure of a local parametrisation for $B_T$ is crucial in proving Proposition~\ref{prop:normal}.
\end{remark}

\noindent The next proposition provides a formula for the unit normal vector to the hypersurface $B_T$ in terms of local parametrisations.

\begin{proposition} \label{prop:normal}
	Let $\{\mathbb{R}^{+} \times U_{\boldsymbol \omega},\mathbb{R}^{+} \times V_{\boldsymbol \omega}, \sigma_{\boldsymbol w} = (\pi_t,{\boldsymbol \gamma}_{\boldsymbol \omega})\}_{\boldsymbol \omega}$ be a local parametrisation of $B_T$ as in Remark~\ref{rem:time_slice} and $\{ \mathcal{O}_{\boldsymbol \omega}, f_{\boldsymbol \omega}\}$ be a $\mathscr{C}^1$--smooth local representation of it in the sense of Definition~\ref{defn:local_rep}, where ${\boldsymbol \omega} = (t,{\boldsymbol z}) \in B_T$. Then, the unit normal to the hypersurface $B_T$ can be expressed as follows:
	\begin{equation}
	{\boldsymbol n}_{B_T} = \dfrac{( -\nabla f_{\boldsymbol \omega}\cdot\partial_t {\boldsymbol \gamma}_{\boldsymbol \omega},\nabla f_{\boldsymbol \omega} )}{\left\vert\left\vert ( -\nabla f_{\boldsymbol \omega}\cdot\partial_t {\boldsymbol \gamma}_{\boldsymbol \omega},\nabla f_{\boldsymbol \omega} ) \right\vert\right\vert}_{2}.
	\label{eqn:normal_bt}
	\end{equation}
	\label{prop_normal}
\end{proposition}
\begin{proof}
	A (non-unit) normal to $B_T$ at the point $(t,{\boldsymbol z}) \in B_T\cap \mathcal{O}_{\boldsymbol \omega}$ can be expressed as  $\nabla_{(t,{\boldsymbol x})}f_{\boldsymbol \omega}(t,{\boldsymbol z}) = (\partial_t f_{\boldsymbol \omega}(t,{\boldsymbol z}),\nabla f_{\boldsymbol \omega}(t,{\boldsymbol z}))$. Definition~\ref{defn:local_param} yields a point $(t,{\boldsymbol y}) \in \mathbb{R}^{+} \times U_{\boldsymbol \omega}$ such that $(t,{\boldsymbol z}) = (t,{\boldsymbol \gamma}_{\boldsymbol \omega}(t,{\boldsymbol y}))$ . Since $f_{\boldsymbol \omega}$ is zero in $B_T\cap  \mathcal{O}_{\boldsymbol \omega}$ the time derivative $\frac{\mathrm{d}}{\mathrm{d}t} f_{\boldsymbol \omega}(t,{\boldsymbol \gamma}(t,{\boldsymbol y}))$ is also zero. Therefore, in $B_T\cap \mathcal{O}_{\boldsymbol \omega}$
	\begin{equation*}
	\partial_t f_{\boldsymbol \omega}(t,{\boldsymbol z}) = -\nabla f_{\boldsymbol \omega}(t,{\boldsymbol z}) \cdot\partial_t {\boldsymbol \gamma}_{\boldsymbol \omega}(t,{\boldsymbol y})
	\end{equation*}
	and a normal to $B_T$ at $(t,{\boldsymbol z})$ is provided by
	\begin{equation*}
	\nabla_{(t,{\boldsymbol x})}f_{\boldsymbol \omega}(t,{\boldsymbol z}) = ( -\nabla f_{\boldsymbol \omega}(t,{\boldsymbol z}) \cdot\partial_t {\boldsymbol \gamma}_{\boldsymbol \omega}(t,{\boldsymbol y}),\nabla f_{\boldsymbol \omega}(t,{\boldsymbol z}) ),
	\end{equation*}
	normalisation of which yields~\eqref{eqn:normal_bt}.
\end{proof}

\begin{remark}
	Since $\{ \mathcal{O}_{\boldsymbol \omega}, f_{\boldsymbol \omega}\}$  is a $\mathscr{C}^1-$smooth local representation of the hypersurface $B_T$, for a fixed time $t$, the unit normal to the boundary $\Gamma(t)$ is given by $-\nabla f_{\boldsymbol \omega}/||\nabla f_{\boldsymbol \omega}||_{2}$.
	\label{rem_normal}
\end{remark}

Next, we present the equivalence between the weak formulations~\ref{it.se1} and~\ref{it.sw1}.

\begin{theorem}[Equivalence]
	\label{thm:eq_thm_1}
	\begin{enumerate}[label= $\mathrm{(ET.\alph*)}$,ref=$\mathrm{(ET.\alph*)}$,leftmargin=\widthof{(ET.a)}+3\labelsep]
		\item\label{it.eta} Let $B_T$ be $\mathscr{C}^1$--regular and  $(\alpha,\velc,p,c,\Omega)$ be a NUM-weak solution. Set $\widetilde{\alpha} := \alpha$, $\velcex := \velc$, $\widetilde{p} := p$ and $\Cex := c$ in $D_T$; $\widetilde{\alpha} := 0$, $\velcex := {\boldsymbol 0}$, $\widetilde{p} := 0$ and $\Cex := 1$ in $\mathscr{D}_T \backslash \overline{D}_T$. If $\alpha \in BV(D_T)$, then $(\alex,\velcex,\widetilde{p},\Cex,\widetilde{\Omega})$ is a NUM-extended solution.
		\item\label{it.etb}  Let $(\alex,\velcex,\widetilde{p},\Cex,\widetilde{\Omega})$ be a NUM-extended solution and assume that $\widetilde{B}_T := \partial \widetilde{D}_T \backslash ([\{0\} \times \Omega(0)] \cup [\{T\} \times \widetilde{\Omega}(T)])$ is $\mathscr{C}^1$--regular, where $\widetilde{D}_T$ is given by~\ref{it.se2} in Definition~\ref{defn:nsm_ext} and $\widetilde{\alpha}_{\vert \widetilde{D}_T} > 0$ on $\widetilde{B}_T$. If there exist constants $0 < \widetilde{m}_{11} \le m_{01}$ and $m_{02} \le \widetilde{m}_{12} < 1$ such that $\widetilde{m}_{11} \le \alex_{\vert\widetilde{D}_T} \le \widetilde{m}_{12}$ and $\alex \in BV(\mathscr{D}_T)$, then $\widetilde{D}_T = D_T$ and $(\alex_{\vert D_T},\velcex_{\vert D_T},\widetilde{p}_{\vert D_T},\Cex_{\vert D_T},\widetilde{\Omega})$ is a NUM-weak solution.
	\end{enumerate}
\end{theorem}

\begin{mproof}
	\begin{enumerate}[label= $\mathrm{(ET.\alph*)}$,leftmargin=\widthof{(ET.a)}+3\labelsep]
		\item \label{et.a} Let $\{\mathbb{R}^{+} \times U_{\boldsymbol \omega},\mathbb{R}^{+} \times V_{\boldsymbol \omega}, \sigma_{\boldsymbol w} = (\pi_t,{\boldsymbol \gamma}_{\boldsymbol \omega})\}_{\boldsymbol \omega}$ be  a local parametrisation of $B_T$. Choose $\widetilde{\varphi}$ belonging to $\mathscr{C}_c^{\infty} \left([0,T)\times \Omega_\ell\right)$. Since $\widetilde{\varphi}_{\vert{D_T}} \in \mathscr{C}_c^{\infty}(\overline{D}_T\backslash (\{T\} \times \Omega(T)))$ and $\widetilde{\alpha} = 0$ in $\mathscr{D}_{T}\backslash \overline{D_{T}}$, the following holds:
		\begin{subequations}
			\begin{multline}
			\int_{D_T} (\alex,\alex\velcex)\cdot \nabla_{(t,{\boldsymbol x})}\widetilde{\varphi}\,\mathrm{d}t\,\vx + \int_{\Omega(0)} \widetilde{\varphi}(0,{\boldsymbol x}) \alpha_0({\boldsymbol x})\,\vx + \int_{D_T} \alex f(\alex,\Cex) \widetilde{\varphi}\,\mathrm{d}t\,\vx  \\
			= \int_{B_T} (\alpha, \alpha {\boldsymbol u})\cdot {\boldsymbol n}_{B_T}\widetilde{\varphi}\,\mathrm{d}s 
			\label{alex_dt} 
			\end{multline}
			and
			\begin{equation}
			\int_{\mathscr{D}_T\backslash D_T} (\alex,\alex\velcex)\cdot \nabla_{(t,{\boldsymbol x})}\widetilde{\varphi}\,\mathrm{d}t\,\vx +  \int_{\mathscr{D}_T\backslash D_T} \alex f(\alex,\Cex) \widetilde{\varphi}\,\mathrm{d}t\,\vx = 0.
			\label{alex_dtc}
			\end{equation}
		\end{subequations}
		A use of Proposition \ref{prop_normal} and Remark \ref{rem_normal} yields
		\begin{equation}
		K_N (\alpha, \alpha {\boldsymbol u})_{\vert B_T}\cdot {\boldsymbol n}_{B_T} = (\alpha, \alpha {\boldsymbol u})_{\vert B_T}\cdot \left(-{\boldsymbol n}_{\vert{\Gamma(t)}} \cdot \partial_t {\boldsymbol \gamma}_{\boldsymbol \omega},{\boldsymbol n}_{\vert{\Gamma(t)}} \right),
		\label{eqn:au_brelation}
		\end{equation}
		where $K_N\not=0$ is a normalisation constant. 
		We then use \eqref{bdrequation} in~\eqref{eqn:au_brelation} to obtain $(\alpha, \alpha {\boldsymbol u})_{\vert B_T}\cdot {\boldsymbol n}_{B_T} = 0$. Add \eqref{alex_dtc} and \eqref{alex_dt} to arrive at~\eqref{eqn:huper_ex_wf}. The conditions on $\widetilde{ \boldsymbol u},\,\widetilde{p}$, and $\widetilde{c}$ follow naturally from Definition~\ref{defn:nsm_ext}.
		\item \label{et.b} Let $\{\mathbb{R}^{+} \times U_{\boldsymbol \omega},\mathbb{R}^{+} \times V_{\boldsymbol \omega}, \sigma_{\boldsymbol w} = (\pi_t,{\boldsymbol \gamma}_{\boldsymbol \omega})\}_{\boldsymbol \omega}$ be  a local parametrisation of $\widetilde{B}_T$. Define a vector field $\vecfd : \mathscr{D}_T \rightarrow \mathbb{R}^{d+1}$ by $\vecfd := (\widetilde{\alpha},\widetilde{\alpha}\velcex)$. For $(t_0,{\boldsymbol x}_0) \in \widetilde{B}_T$, define 
		\begin{equation}
		\vecfd_{\vert{\widetilde{B}_T^{+}}}(t_0,{\boldsymbol x}_0) := \displaystyle \lim_{\tiny \begin{array}{c}
			(t,{\boldsymbol x}) \rightarrow (t_0,{\boldsymbol x}_0) \\
			(t,{\boldsymbol x}) \in \widetilde{D}_T
			\end{array}}{\boldsymbol F}(t,{\boldsymbol x}),\;\; \vecfd_{\vert{\widetilde{B}_T^{-}}}(t_0,{\boldsymbol x}_0) := \displaystyle \lim_{\tiny \begin{array}{c}
			(t,{\boldsymbol x}) \rightarrow (t_0,{\boldsymbol x}_0) \\
			(t,{\boldsymbol x}) \in \mathscr{D}_T\backslash \overline{\widetilde{D}}_T
			\end{array}}{\boldsymbol F}(t,{\boldsymbol x}).
		\end{equation}
		The fact that $\vecfd = {\boldsymbol 0}$ in $\mathscr{D}_T\backslash \overline{\widetilde{D}}_T$ (since $\alex_{\vert \mathscr{D}_T\backslash \overline{\widetilde{D}}_T} = 0$ from (SE.2)) yields $\vecfd\vert_{\widetilde{B}_T^{-}} = {\boldsymbol 0}$ and hence,  
		\begin{equation}
		\int_{\widetilde{B}_T} \varphi(\widetilde{\alpha},\widetilde{\alpha}\velcex)_{\vert \widetilde{D}_T}\cdot {\boldsymbol n}_{\widetilde{B}_T}\,\mathrm{d}s = \int_{\widetilde{B}_T}  \left(\vecfd_{\vert{\widetilde{B}_T^{+}}} - \vecfd_{\vert{\widetilde{B}_T^{-}}}\right)\cdot{\boldsymbol n}_{\widetilde{B}_T} \varphi \mathrm{d}s.
		\label{eqn:bvalues_etb}
		\end{equation}
		Since the weak divergence of $\vecfd$ given by $-\widetilde{\alpha}f(\widetilde{\alpha},\widetilde{c})$ belongs to $L^2(\mathscr{D}_T)$, the normal jump $(\vecfd_{\vert{\widetilde{B}_T^{+}}} - \vecfd_{\vert{\widetilde{B}_T^{-}}})\cdot\boldsymbol{n}_{\widetilde{B}_T} $ is zero. Consequently,  $(\widetilde{\alpha},\widetilde{\alpha}\velcex)_{\vert \widetilde{D}_T}\cdot {\boldsymbol n}_{\widetilde{B}_T} = 0$ on $\widetilde{B}_T$. Then, the fact that $\widetilde{\alpha}_{\vert \widetilde{D}_T} > 0$ on $\widetilde{B}_T$, Proposition~\ref{prop_normal}, and Remark~\ref{rem_normal} yield
		\begin{equation}
		\partial_t \widetilde{\boldsymbol \gamma}_{\boldsymbol \omega}\cdot{\boldsymbol n}_{\vert{\widetilde{\Gamma}(t)}} = \velcex_{\vert{\widetilde{\Gamma}(t)}}\cdot{\boldsymbol n}_{\vert{\widetilde{\Gamma}(t)}}.
		\label{eqn:boundary_mod}
		\end{equation}
		Since $\widetilde{\boldsymbol \gamma}_{\boldsymbol \omega}(0,\cdot) = {\boldsymbol \gamma}_{\boldsymbol \omega}(0,\cdot)$,~\eqref{eqn:boundary_mod} yields  $\widetilde{D}_T = D_T$. Choose $\widetilde{\varphi} \in \mathscr{C}_c^\infty(\overline{D_T}\backslash(\{T\} \times \Omega(T)))$. Define $\varphi \in \mathscr{C}_c^{\infty}([0,T) \times \Omega_\ell)$ such that $\varphi = \widetilde{\varphi}$ in $D_T$. Since $\widetilde{D}_T = D_T$ and  $\widetilde{\alpha}= 0$ on $\mathscr{D}\backslash \overline{D}_T$,~\eqref{eqn:boundary_mod} yields
		\begin{multline}
		\int_{D_T} (\alex,\velcex \alex)\cdot \nabla_{(t,{\boldsymbol x})}\widetilde{\varphi}\,\mathrm{d}t\,\vx  +  \int_{\Omega(0)} \widetilde{\varphi}(0,{\boldsymbol x}) \alpha_0({\boldsymbol x})\,\vx + \int_{D_T}\alex f(\alex,\Cex) \widetilde{\varphi}\,\mathrm{d}t\,\vx \nonumber \\
		=  \int_{B_T} \widetilde{\varphi}(\alex,\velcex \alex)\cdot {\boldsymbol n}_{B_T}\,\mathrm{d}s.
		\end{multline}
		Therefore,  $\widetilde{\alpha}_{\vert D_T}$ satisfies $\eqref{hyper_wf}$. The conditions on $\velcex_{\vert D_T},\,\widetilde{p}_{\vert D_T},\,\textrm{ and }\Cex_{\vert D_T}$ follow from Definition~\ref{defn:nsm_weak}. \qed
	\end{enumerate}
\end{mproof}

\begin{remark}
	The properties that $\alpha \in BV(D_T)$ and $\widetilde{\alpha} \in BV(\mathscr{D}_T)$ are necessary in the proof of~\ref{et.a} and~\ref{et.b}, respectively so that the boundary values in~\eqref{eqn:au_brelation} and~\eqref{eqn:bvalues_etb} are well defined in sense of traces (see Theorem 1~\cite[p.~177]{Evans20151}).
\end{remark}

\section{Numerical scheme}
\label{sec:n_scheme}
\subsection{Discretisation}
Here, we consider for simplicity that the spatial dimension is equal to 2.
The temporal domain $[0,T]$ is uniformly partitioned into $N$ intervals, $\mathcal{T}_n = (t_{n},t_{n+1})$, with $\delta = t_{n+1} - t_{n}$ for $n = 0,\ldots,N-1$, where $t_0 = 0$ and $t_{N} = T$. Let $\mathscr{T} = \{K_{j}\}_{j=1,\ldots,J}$ be a conforming Delaunay partition of the domain $\Omega_\ell$ into triangles. The following notations will be followed in the sequel. For $i, j = 1,\ldots,J$,  
\begin{itemize}
	\item ${\boldsymbol z}_j$: centroid of the $K_j$, $a_j$:  area of the $K_j$,
	\item $\mathcal{E}(j)$: set of all triangles sharing a common edge with $K_j$; $\mathcal{V}(j)$: set of all vertices of a triangle $K_j$,
	\item $e_{ji}$: common edge between triangles $K_j$ and $K_i$; ${\boldsymbol m}_{ji}$: mid point of $e_{ji}$; ${\boldsymbol n}_{ji}$:  unit normal to the edge $e_{ji}$ pointing from the triangle $K_j$; $\ell_{ji}$:  length of $e_{ji}$, 
	\item $\mathscr{V} = ({\boldsymbol v}_j)_{j=1,\ldots, M}$: collection of vertices of triangles in $\mathscr{T}$,
	\item $\mathscr{B}_e$: set of all boundary edges in $\mathscr{T}$; and $\mathscr{B}_T$: set of all boundary triangles.
\end{itemize}

\begin{definition}[Discrete average]
	For any real valued function $f$ on $\mathbb{R}^2$, define the discrete average of $f$ on the triangle $K_j$ by $\{\!\!\{f\}\!\!\}_{K_j} := \sum_{{\boldsymbol v}_i \in \mathcal{V}_j} f({\boldsymbol v}_i)/3$, where $j = 1,\ldots,J$.
\end{definition}

The following aspects need to be considered when choosing a proper triangulation for $\Omega_\ell$.
\subsubsection{Mesh-locking effect}
\label{sec:mesh_lcoking}
We use a finite volume scheme to approximate the hyperbolic conservation law~\eqref{eqn:v_frac}, and it is a well-known fact that finite volume solutions exhibit the mesh-locking effect, see \cite{EGM13} and references therein. That is, the computed solution is preferentially oriented in accordance with the orientation of the triangulation. Further, the domain $\widetilde{\Omega}(t)$ obtained from~\ref{it.se1} in Definition~\ref{defn:nsm_ext}, depends on $\widetilde{\alpha}$. Therefore, the mesh-locking effect in $\widetilde{\alpha}$ at the discrete level affects the accuracy of $\widetilde{\Omega}$, and thus other variables as well. This error propagates at each time step in a compounding  fashion. One way to eliminate this problem is to use a very refined triangulation, but this increases the computational cost. The natural and cost-effective way is to use an unstructured and random triangulation. Randomness avoids any particular orientation of the triangles and thus eliminates mesh-locking from the numerical solution.

\subsubsection{Approximation of the initial domain}
\label{sec:approximation}
After triangulating  $\Omega_\ell$, we approximate the initial domain $\Omega(0)$ by the set $\Omega_h^0$ where,
\begin{equation}
\Omega_h^0 := \cup_{\{{\boldsymbol z_j} \in \Omega(0) \}} K_j. 
\label{eqn:appr_omega_o}
\end{equation}
However, this approximation of $\Omega(0)$ by $\Omega_h^0$ is not accurate if the triangles are arranged in a structured manner. We illustrate this in Figure~\ref{fig:struct_approx}, where $\Omega(0)$ - a circle centred at the origin with unit radius is approximated by $\Omega_h^0$ in different structured triangulations. Evidently, the coarse triangulations in Figures~\ref{fig:ref_5} and~\ref{fig:ref_6} with 1024 and 4096 triangles, respectively give a poor approximation of $\Omega(0)$. A reasonably good approximation is provided by the triangulation in~Figure~\ref{fig:ref_7}; however, this triangulation contains 16,384 triangles, which makes the computations expensive over multiple time steps. If the discrete approximation of $\Omega(0)$ is not smooth enough, the discrete solution loses its symmetry as time evolves and this phenomenon is observed in the work by M. E. Hubbard and H. M. Byrne~\cite{hubbard}. 
\begin{figure}[h!]
	\centering
	\hspace{0.1cm}	\begin{subfigure}[b]{0.32\textwidth}
		\includegraphics[scale=1]{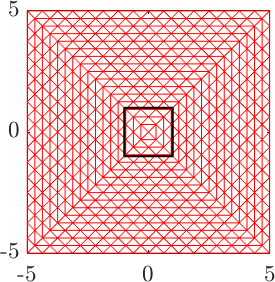}
		\caption{}
		\label{fig:trian_st_5}
	\end{subfigure}
	\begin{subfigure}[b]{0.32\textwidth}
		\includegraphics[scale=1]{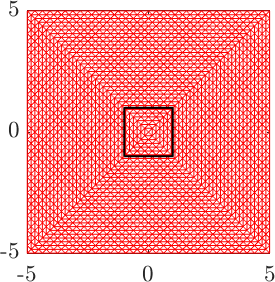}
		\caption{}
		\label{fig:trian_st_6}
	\end{subfigure}
	\begin{subfigure}[b]{0.32\textwidth}
		\includegraphics[scale=1]{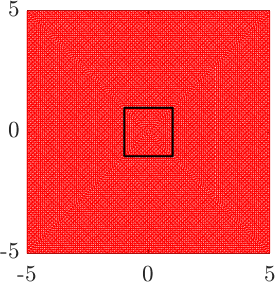}
		\caption{}
		\label{fig:trian_st_7}
	\end{subfigure}
	\newline
	\centering
	\begin{subfigure}[b]{0.31\textwidth}
		\includegraphics[scale=1]{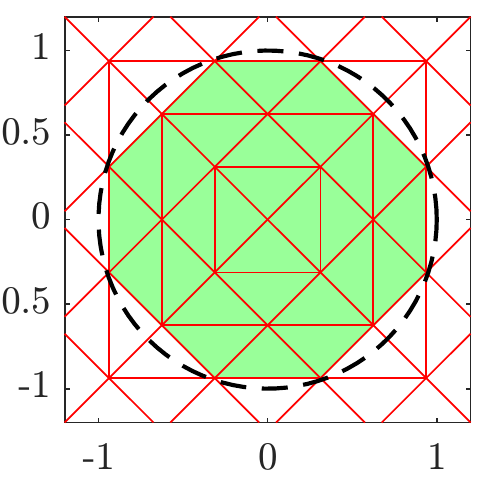}
		\caption{}
		\label{fig:ref_5}
	\end{subfigure}
	~ 
	\begin{subfigure}[b]{0.31\textwidth}
		\includegraphics[scale=1]{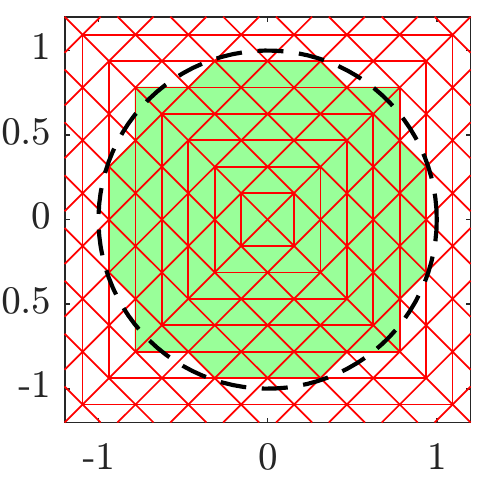}
		\caption{}
		\label{fig:ref_6}
	\end{subfigure}
	~ 
	\begin{subfigure}[b]{0.31\textwidth}
		\includegraphics[scale=1]{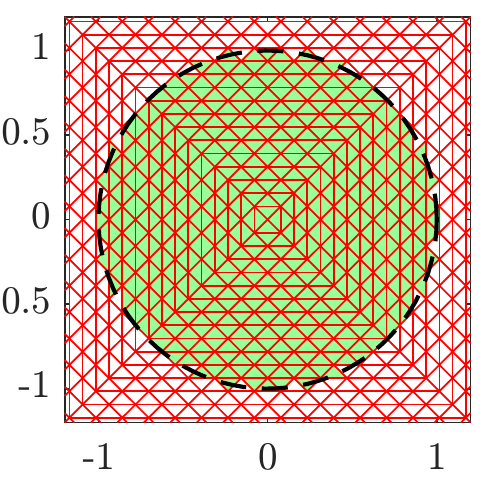}
		\caption{}
		\label{fig:ref_7}
	\end{subfigure}
	\caption{First row: Figures~\ref{fig:trian_st_5},~\ref{fig:trian_st_6} and~\ref{fig:trian_st_7} are structured triangulations of the domain $\Omega_\ell = (-5,5)^2$. Triangulations in ~\ref{fig:ref_5},~\ref{fig:ref_6}, and~\ref{fig:ref_7}, respectively contain 1024, 4096, and 16,384 triangles. Second row: Here, axes are limited to the region $(-1,1)^2$ (black box in the first row) and corresponding approximations (green region) of an initial domain in the shape of a circle centred at origin with unit radius. }\label{fig:struct_approx}
\end{figure}

\noindent We overcome the issues discussed in Subsections~\ref{sec:mesh_lcoking} and~\ref{sec:approximation} by using an adaptive and random triangulation. In particular, we employ the mesh generation of \emph{Ruppert's algorithm} put forward by J. Ruppert~\cite{Ruppert1995548}. Ruppert's algorithm is based on Delaunay refinements, and produces quality triangulations without any skinny triangles; that is every angle in a triangle is greater than a preset value $\theta_{\min}$. To obtain a good approximation of the domain $\Omega(0)$, we specify a finite number of nodes $\mathscr{N} = ({\boldsymbol N}_i)_{1\le i \le N_{0}}$ (in anti-clockwise order) on $\partial \Omega(0)$, join the neighbouring nodes ${\boldsymbol N}_{i}$ and ${\boldsymbol N}_{i+1}$ by a straight line segment denoted by ${\boldsymbol N}_{i,i+1}$, and let this collection of straight edges be denoted by $\mathscr{L}(\mathscr{N})$. This procedure gives a piecewise affine approximation of $\partial \Omega(0)$. Ruppert's algorithm constructs a triangulation such that corresponding to each straight edge  ${\boldsymbol N}_{i,i+1} \in  \mathscr{L}(\mathscr{N})$, there exists a triangle $K_j$ such that ${\boldsymbol N}_{i,i+1}$ is an edge of $K_j$.
\begin{figure}[h!]
	\centering
	\hspace{0.1cm}
	\begin{subfigure}[b]{0.3\textwidth}
		\includegraphics[scale=0.95]{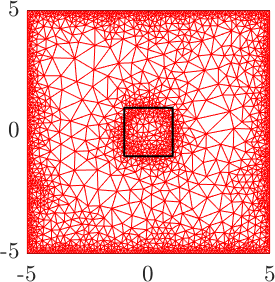}
		\caption{}
		\label{fig:tria_circle}
	\end{subfigure}
	~ 
	\begin{subfigure}[b]{0.3\textwidth}
		\includegraphics[scale=0.95]{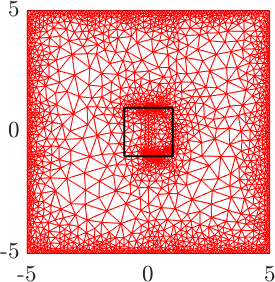}
		\caption{}
		\label{fig:tria_bullet}
	\end{subfigure}
	~ 
	\begin{subfigure}[b]{0.3\textwidth}
		\includegraphics[scale=0.95]{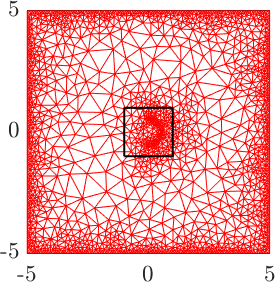}
		\caption{}
		\label{fig:tria_donut}
	\end{subfigure}
	\newline
	\begin{subfigure}[b]{0.3\textwidth}
		\includegraphics[scale=0.97]{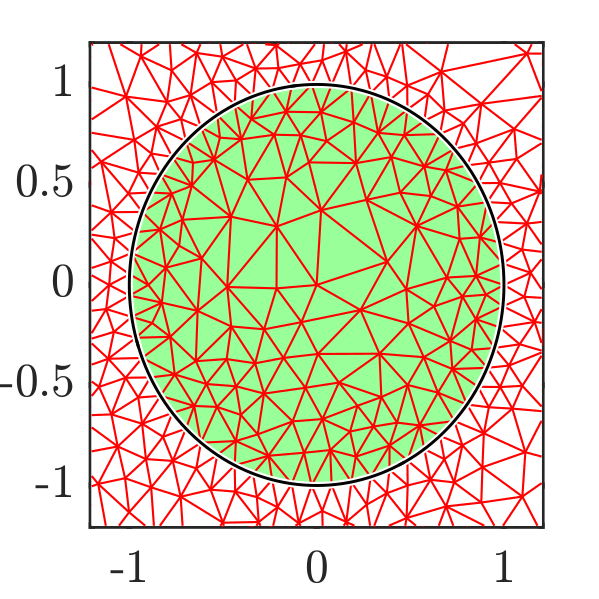}
		\caption{}
		\label{fig:rand_c}
	\end{subfigure}
	~ 
	\begin{subfigure}[b]{0.3\textwidth}
		\includegraphics[scale=0.97]{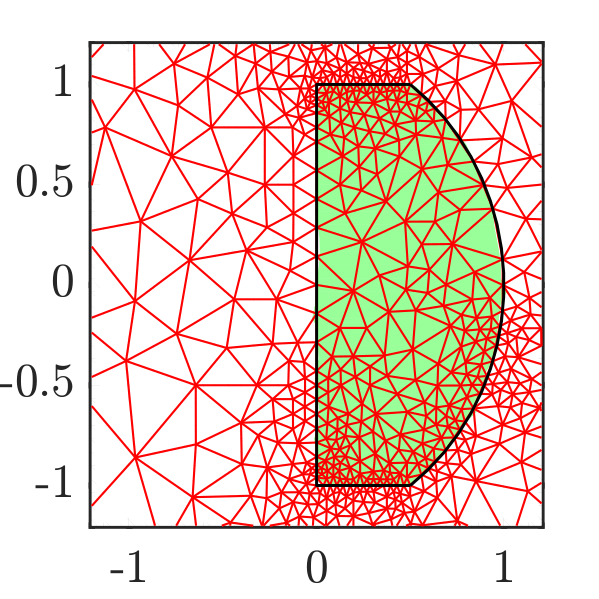}
		\caption{}
		\label{fig:rand_b}
	\end{subfigure}
	~ 
	\begin{subfigure}[b]{0.3\textwidth}
		\includegraphics[scale=0.97]{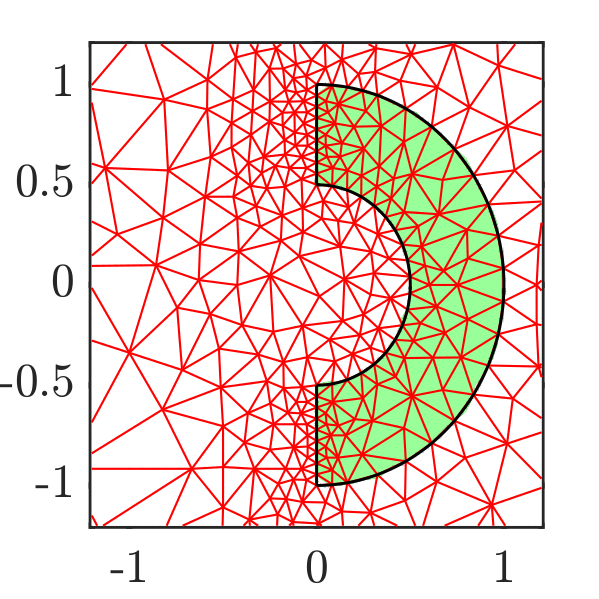}
		\caption{}
		\label{fig:rand_d}
	\end{subfigure}
	\caption{First row: Figures~\ref{fig:tria_circle},~\ref{fig:tria_bullet}, and~\ref{fig:tria_donut} are the unstructured (Ruppert-Delaunay) triangulations of the domain $\Omega_\ell = (-5,5)^2$ corresponding to initial domains with circular, bullet, and semi--annular shapes, respectively. The triangulations in~\ref{fig:tria_circle},~\ref{fig:tria_bullet}, and~\ref{fig:tria_donut}, respectively, contain 3492, 3642 and 4084 triangles. 
		Second row: Here, axes are limited to the region $(-1,1)^2$ (black box in the first row) to display the initial approximations (green region) better. }\label{fig:unstruct_approx}
\end{figure}
These aspects of Ruppert's algorithm help us to obtain a good approximation of $\Omega(0)$ irrespective of its shape. The fact that the algorithm uses reasonably few number of triangles is an added advantage. In Figure~\ref{fig:unstruct_approx}, we show the approximation of $\Omega(0)$ by $\Omega_h^0$, where the triangulations are obtained by Ruppert's algorithm. The circular, bullet-shaped and semi-annulus shaped domains, respectively shown in Figures~\ref{fig:rand_c},~\ref{fig:rand_b}, and~\ref{fig:rand_d}; are well approximated by the corresponding triangulations. In each case, we require fewer than 4100 triangles to obtain a good approximation of $\Omega(0)$ as opposed to $16,384$ triangles in the case of a structured triangulation (see Figure~\ref{fig:ref_7}). This illustrates the economical advantage of Ruppert's algorithm.

Next, we present the numerical scheme. We discretise~\eqref{eqn:v_frac} using a finite volume method,~\eqref{eqn:velocity}-\eqref{eqn:pressure} using Lagrange $\mathbb{P}_2-\mathbb{P}_1$ Taylor-Hood finite element method and~\eqref{eqn:o_tension} using a backward Euler in time and $\mathbb{P}_1$ mass lumped finite element method.

\begin{definition}[Discrete scheme for the NUM model] \label{defn:discrete_scheme}
	Define
	\begin{itemize}
		\item $\alpha_h^0$ by $\alpha_h^0 := \alpha_j^0$ on $K_j$, for $j=1,\ldots,J$, where $\alpha_j^0 := \fint_{K_j}\alpha_0(\boldsymbol{x})\,\vx$.
		\item $c_{h}^0$ by $c_{h\vert{K_j}}^0 \in \mathbb{P}_1(K_j)$ for $j=1,\ldots,J$, where $c_{h}^0({\boldsymbol v}_i) = c_0({\boldsymbol v}_i)$ for $i = 0,\ldots,M$.
		\item $\Omega_h^0$ is given by~\eqref{eqn:appr_omega_o}.
	\end{itemize}
	Fix a threshold $\athr \in (0,1)$ and $\Omega_\ell$ such that $\Omega_{h}^0 \subset \Omega_\ell$. The function ${\boldsymbol u}_h^0$ is obtained from~\ref{it.dsc} by taking $n= 0$. Construct a finite sequence of 4-tuple of functions $(\alpha_h^n,{\boldsymbol u}_h^n,p_h^n,c_{h}^n)_{\{1\le n\le N\}}$ on $\Omega_\ell$ such that 
	for all $1 \leq n \leq N$, \ref{it.dsa}--\ref{it.dsd} hold.
	\begin{enumerate}[label= $\mathrm{(DS.\alph*)}$,leftmargin=\widthof{(DS.a)}+3\labelsep]
		\item\label{it.dsa}  $\alpha_h^n := \alpha_j^n$ on $K_j$ for $j=1,\ldots,J$, where
		\begin{multline}
		\dfrac{1}{\delta}(\alpha_j^n - \alpha_j^{n-1}) + \dfrac{1}{a_j} \sum_{e_{ji} \in \mathcal{E}(j)} \ell_{ji} \mathcal{F}_{ji}^{n-1} \\
		= 
		(\alpha_j^{n-1} - \athr)^{+}(1 - \alpha_j^{n-1})b_{j}^{n-1} - (\alpha_j^{n} - \athr)^{+}d_{j}^{n-1},
		\label{eqn:disc_vf}
		\end{multline}
		where, $\mathcal{F}_{ji}^{n-1}$ is the upwind flux between the triangles $K_j$ and $K_i$ through the common edge $e_{ji}$ defined by 
		\begin{equation}
		\label{eqn:n_flux}
		\mathcal{F}_{ji}^{n} := (\velc_{ji}^n\cdot{\boldsymbol n}_{ji})^{+}\alpha_j^n - (\velc_{ji}^n\cdot{\boldsymbol n}_{ji})^{-}\alpha_i^n,
		\end{equation}
		$\velc_{ij}^n = \velc_h^n({\boldsymbol m}_{ji})$, $b_j^n = \{\!\!\{(1 + s_1)c_h^n/(1 + s_1c_h^n)\}\!\!\}_{K_j}$, and $d_j^n = \{\!\!\{(s_2 + s_3)c_h^n/(1 + s_4c_h^n)\}\!\!\}_{K_j}$. If $e_{ji} \in \mathscr{B}_e$, then we set $\alpha_{i}^n$ to zero. This choice is justified since $\velc_{ji}^n = {\boldsymbol 0}$, so any choice of $\alpha_i^n$ does not change the value of the flux.
		\item\label{it.dsb} 
		$\Omega_h^n$ is defined through the following process: starting from $\Omega_h^{n-1}$,
		\begin{itemize}
			\item[(1)] add all triangles $K_j\not\subset\Omega_h^{n-1}$ that have an edge on $\partial \Omega_h^{n-1}$ and such that $\alpha_j^n\ge \athr$;
			\item[(2)] remove all triangles $K_j\subset \Omega_h^{n-1}$ that have an edge on $\partial\Omega_h^{n-1}$ and such that $\alpha_j^n<\athr$;
			\item[(3)] Steps (1) and (2) lead to a new domain $U$; repeat (2) with $U$ instead of $\Omega_h^{n-1}$ until all triangles $K_j$ that have an edge on $\partial U$ satisfy $\alpha_j^n\ge \athr$, and define $\Omega_h^n$ as the resulting final set $U$.
		\end{itemize}
		
		\item\label{it.dsc} Set the conforming finite element space of piecewise second degree polynomials from $\Omega_h^n$ to $\mathbb{R}^2 $ with homogeneous tangential component on $\partial \Omega_h^n$ by 
		\begin{align}
		\boldsymbol{W}_{h,0}^n := \left\{ {\boldsymbol \varphi}_h^n \in (\mathscr{C}^0(\overline{\Omega_h^n}))^2 : {\boldsymbol \varphi}_{h\vert K_j}^n\in (\mathbb{P}_2(K_j))^2 \;\forall K_j \subset \Omega_h^n,\, {\boldsymbol \varphi}_{h\vert \partial \Omega_h^n}^n\cdot{\boldsymbol \tau}_{\vert \partial \Omega_h^n} = 0 \right\}. 
		\end{align}
		Set the conforming finite element space of piecewise linear polynomials from $\Omega_h^n$ to $\mathbb{R}$ and its subspace with homogeneous Dirichlet boundary condition on $\partial \Omega_h^n$ by
		\begin{align}
		S_h^n &:= \left\{ v_h^n \in \mathscr{C}^0(\overline{\Omega_h^n}) : v_{h\vert K_j}^n \in \mathbb{P}_1(K_j) \;\forall K_j \subset \Omega_h^n \right\} \text{ and }\\
		S_{h,0}^n &:= \left\{ v_h^n \in S_{h}^n,\, v_{h\vert \partial \Omega_h^n}^n = 0 \right\}.
		\end{align}
		Then,
		\begin{equation}
		\velc_h^n := \left\{ \begin{array}{c l}
		\velcex_h^n & \text{ on } \Omega_h^n, \\
		{\boldsymbol 0} & \text{ on } \Omega_\ell\backslash \overline{\Omega_h^n}
		\end{array}
		\right. \text{ and }
		p_h^n := \left\{ \begin{array}{c l}
		\widetilde{p}_h^n & \text{ on } \Omega_h^n, \\
		0& \text{ on } \Omega_\ell\backslash \overline{\Omega_h^n},
		\end{array}
		\right.
		\end{equation}
		where $(\velcex_h^n,\widetilde{p}_h^n) \in {\boldsymbol W}_{h,0}^n \times S_{h,0}^n$ satisfies, for all $\varphi_h^n \in {\boldsymbol W}_{h,0}^n$ and $v \in S_{h,0}^n$,
		\begin{align}
		\label{eqn:pvsys_dis_1}
		a_{1,h}^n(\velcex_h^n,{\boldsymbol \varphi}_h^n) - a_{3,h}^n(\widetilde{p}_h^n,{\boldsymbol \varphi}_h^n) &=    \mathcal{L}_h^n({\boldsymbol \varphi}_h^n), \\
		\label{eqn:pvsys_dis_2}
		a_{2,h}^n(\widetilde{p}_h^n,v_h^n) + a_{3,h}^n(v_h^n,\velcex_h^n) &= 0,
		\end{align}
		with $a_{1,h}^n : {\boldsymbol W}_{h,0}^n \times {\boldsymbol W}_{h,0}^n \rightarrow \mathbb{R},\,a_{2,h}^n : S_{h,0}^n \times {\boldsymbol W}_{h,0}^n \rightarrow \mathbb{R},\, a_{3,h}^n : S_{h,0}^n \times S_{h,0}^n \rightarrow \mathbb{R}$ and $\mathcal{L}_{h}^n : {\boldsymbol W}_{h,0}^n \rightarrow \mathbb{R}$ are defined by
		\begin{align}
		\label{eqn:a_1hn}
		a_{1,h}^n({\boldsymbol u},{\boldsymbol v}) &= \int_
		{\Omega_h^n} \alpha_h^n\left(2\mu \nabla_s{\boldsymbol u}:\nabla_s{\boldsymbol v} + \lambda \mathrm{div}({\boldsymbol u})\mathrm{div}({\boldsymbol v})\right)\vx, \\
		\label{eqn:a_3hn}
		a_{2,h}^n(p,z) &= \int_{\Omega_h^n} \dfrac{1-\alpha_h^n}{k\alpha_h^n} \nabla p \cdot \nabla z\,\mathrm{d}{\boldsymbol x},  \\
		\label{eqn:a_2hn}
		a_{3,h}^n(z,{\boldsymbol w}) &= \int_{\Omega_h^n}  z\,\mathrm{div}({\boldsymbol w})\,\vx, \text{ and } \\
		\label{eqn:ll_hn}
		\mathcal{L}_{h}^n({\boldsymbol v}) &= \int_{\Omega_h^n} \mathscr{H}(\alpha_h^n)\mathrm{div}({\boldsymbol v})\,\vx.
		\end{align}
		\item\label{it.dsd}  Define the finite dimensional vector space of piecewise constant functions
		\begin{align}
		S_{h,ML} := \left\{ w_h : w_h = \sum_{j=1}^M w_j {\boldsymbol \chi}_{\widetilde{K}_j},\,w_j \in \mathbb{R}, 1 \le j \le M \right\},
		\label{eqn:dtspace_ox}
		\end{align}
		where, $\widetilde{K}_j$ is the convex polygon at the vertex ${\boldsymbol v}_j$ defined by 
		\begin{equation}
		\widetilde{K}_j = \left\{ {\boldsymbol x} : {\boldsymbol x} = \sum_{\{i \,:\, {\boldsymbol v}_j \in \overline{K_i}\}} \lambda_i {\boldsymbol z}_i,\, 0 \le \lambda_i \le 1,\,\sum_{i}\lambda_i = 1 \right\}.
		\end{equation}
		The mass lumping operator $\Pi_h : \mathscr{C}^0(\overline{\Omega_\ell}) \rightarrow S_{h,ML}$ is defined by $\Pi_h w = \sum_{j=1}^M w({\boldsymbol v}_j){\boldsymbol \chi}_{\widetilde{K}_j}$. Then, 
		\begin{eqnarray}
		c_{h}^n := \left\{ \begin{array}{c l}
		\widetilde{c}_h^n & \text{ on } \Omega_h^n,\\
		1 & \text{ on } \Omega_\ell \backslash \overline{\Omega_h^n},
		\end{array}
		\right.
		\end{eqnarray}
		where $\widetilde{c}_h^n \in S_{h}^n$ satisfies $\widetilde{c}_{h\vert \partial \Omega_h^n}^n = 1$ and, with $\Pi_h \widetilde{c}_h^n := (\Pi_h c_h^n)_{\vert \Omega_h^n}$,
		\begin{multline}
		\int_{\Omega_{h}^n} \left( \Pi_h\widetilde{c}_h^n - \Pi_h c_h^{n-1} \right) \Pi_hv_h^n\;\vx + \delta\int_{\Omega_{h}^n} \eta\nabla \widetilde{c}_{h}^n\cdot \nabla v_h^n\;\vx  \\
		= -\delta\int_{\Omega_h^n}\dfrac{Q \alpha_{h}^n}{1+\widehat{Q}_1\Pi_h c_h^{n-1}} \Pi_h \widetilde{c}_h^n\Pi_hv_h^n\vx\quad \forall v_h^n \in S_{h,0}^n.
		\label{eqn:c_tn}
		\end{multline}
	\end{enumerate}
\end{definition}

\begin{remark}[Scheme for the NLM model]
	Step~\ref{it.dsd} needs to be modified in the case of numerical experiments for the NLM. In particular, we replace $\Omega_{h}^n$ in~\eqref{eqn:c_tn} by $\Omega_\ell=(-\ell,\ell)^2$ and $\widetilde{c}_h^n$ by $c_h^n$ to incorporate the evolution of the nutrient in the entire domain $\mathscr{D}_T$. Now, the boundary conditions are imposed on $\partial \mathscr{D}_T$, and represent the supply of nutrient through blood vessels at the boundary of the domain.
\end{remark}

\begin{remark}[Determining $\Omega_h^n$]
	The step~\ref{it.dsb} determines the tumour domain. The volume fraction of tumour cells outside $\Omega_{h}^n$ is numerically close to zero while it is significant on the boundary of $\Omega_{h}^n$. That is the boundary of $\Omega_{h}^n$ is the interface beyond which the cell volume fraction reduces to a numerically small value. However, we allow the volume fraction of the tumour cells to become close to zero in some internal parts of $\Omega_{h}^n$, and still remain as integral parts of $\Omega_{h}^n$. 
	
	To ensure the stability of the finite volume discretisation of \eqref{eqn:v_frac}, the time stepping used in simulations must be chosen so that the CFL condition holds; as a consequence, the tumour can only grow by one layer of triangles at each time step, which justifies the choice in Step (1) in~\ref{it.dsb}.
	Additionally, in our simulations we noticed that multiple iterations of Step (2) in~\ref{it.dsb} are not required: after one iteration only, all the resulting boundary triangles have a tumour volume fraction larger than $\athr$.
	\label{rem:det_omeghn}
\end{remark}

\begin{remark}[3D setting]
	The discrete schemes presented here in 2D for the NUM and NLM models extend in a straightforward way to three-dimensional models, since they are based on methods (finite volume, finite elements) that can be applied to 2D and 3D equations, and have the same presentation in both dimensions.
\end{remark}

\begin{definition}[Discrete solution for the NUM model]
	The 4-tuple of functions  $(\alpha_{h,\delta},\velc_{h,\delta},p_{h,\delta},c_{h,\delta})$ defined by  $(\alpha_{h,\delta},\velc_{h,\delta},p_{h,\delta},c_{h,\delta}) := (\alpha_h^n,{\boldsymbol u}_h^n,p_h^n,c_{h}^n)$ on $\mathcal{T}_n$ for $0 \leq n \le N-1$, where the finite sequence $(\alpha_h^n,{\boldsymbol u}_h^n,p_h^n,c_{h}^n)_{\{0 \leq n \leq N-1\}}$ is obtained from Definition~\ref{defn:discrete_scheme} is said to be the discrete solution of the NUM model~\eqref{sys:model_common}--\eqref{eqn:in_cond_NUM_c} with respect to the time discretisation $(\mathcal{T}_n)_{n=0,\ldots,N-1}$ and the triangulation $\mathscr{T}$.
\end{definition}

A few aspects of the numerical scheme need to be discussed briefly. For more details, the reader may refer to~\cite{Remesan2019}.

\subsubsection{Threshold value}
The threshold value $\athr \in (0,1)$ plays an important role in obtaining accurate numerical solutions. The finite volume method used in~\ref{it.dsa} introduces significant numerical diffusion while computing $\alpha_h^n$, due to upwinding of the fluxes. If we define the discrete domain $\Omega_h^n$ as the union of all triangle $K_j$ with $\alpha_{h\vert K_j}^n > 0$, the domain $\Omega_h^n$ might be significantly larger than the exact domain $\Omega(t_n)$. Since the computation of ${\boldsymbol u}_h^n, p_h^n$ and $c_h^n$ depends crucially on $\Omega_h^n$, the error in $\Omega_h^n$ affects the accuracy of these functions as well. Further, $\alpha_{h}^{n+1}$ depends on ${\boldsymbol u}_h^n, p_h^n$, and $c_h^n$. So the error propagates over time steps, finally reducing the quality of numerical solutions significantly. To avoid this, we compare $\alpha_h^n$ with a small positive number, $\athr$. The tumour boundary $\partial \Omega_h^n$ is the polygonal curve constituted by the edges of triangles in $\mathscr{T}$ such that $\alpha_j^n \ge \athr$ in the boundary triangles $K_j$ internal to $\Omega_h^n$, and $\alpha_j^n < \athr$ in every triangle external to $\partial \Omega_h^n$.  However, the triangles in $\Omega_\ell\backslash \Omega_h^n$ have volume fraction in the range $(0,\athr)$. This residual volume fraction causes a spurious growth from the term $\alpha f(\alpha,c)$ in the right hand side of~\eqref{eqn:v_frac} and this effect is eliminated by modifying $\alpha f(\alpha,c)$ to $(\alpha - \athr)^{+}f(\alpha,c)$ in the right hand side of~\eqref{eqn:disc_vf}.
\subsubsection{Numerical methods}
The volume fraction equation~\eqref{eqn:v_frac} is a hyperbolic conservation law. Therefore, we use a finite volume scheme with piecewise constant solutions on each triangle $K_j$. The piecewise constant solutions $\alpha_h^n$ have the added advantage of easy computation of the integrals in~\eqref{eqn:a_1hn}--\eqref{eqn:ll_hn}. The Lagrange $\mathbb{P}_2-\mathbb{P}_1$ Taylor-Hood method ensures the stability of the solutions $(\velc_h^n,p_h^n)$ obtained from~\ref{it.dsc}; note that when $\alpha_h^n$ approaches unity,~\eqref{eqn:velocity} and \eqref{eqn:pressure} become a Stokes system. Moreover,  taking the values of $\velc_h^n$ at the edge mid points facilitates a straight forward computation of the numerical flux defined by~\eqref{eqn:n_flux}. The backward in time Euler method ensures the stability of the numerical solutions $c_h^n$ obtained from~\ref{it.dsd}. The mass lumping $\mathbb{P}_1$ finite element method and the Delaunay based triangulation are used to obtain the positivity and boundedness (by unity) of $c_h^n$~\cite{Thomee200811}.

\section{Numerical results}
\label{sec:n_results}
The tests conducted in this section are categorised into two sets, Set-NUM and Set-NLM, corresponding to NUM and NLM models. The values of the parameters that remain the same in Set-NUM and Set-NLM are tabulated in Table~\ref{tab:tabvalues}.
\begin{table}[h!] 
	\centering
	\begin{tabular}{|c|c||c|c|}
		\hline
		\textbf{Parameter} & \textbf{Value} & \textbf{Parameter} & \textbf{Value} \\ \hline \hline
		$\delta$          & 0.1            & $\mu$             & 1              \\ \hline
		$s_1, s_4$             & 10              & $\lambda$         & -2/3           \\ \hline
		$s_2, s_3$        & 0.5            & $\athr$           & 0.01           \\ \hline
		$\widehat{Q}$             & 0             & $\alpha^{\ast}$   & 0.8            \\ \hline
	\end{tabular}
	\caption{Dimensionless parameters used in the numerical experiments for Set-NUM and Set-NLM.}
	\label{tab:tabvalues}
\end{table}
The numerical values in Table~\ref{tab:tabvalues} are adapted from~\cite{breward_2002} in which a similar model in one spatial dimension is considered. Values of the parameters $Q$ and $\eta$ depend on specific cases and are provided in the later experiments. In all sets of experiments, the initial volume fraction is given by $\alpha(0,{\boldsymbol x}) = 0.8$ when ${\boldsymbol x} \in \Omega_{h}^0$ and $\alpha(0,{\boldsymbol x}) = 0$ when ${\boldsymbol x} \not\in \Omega_{h}^0$ and the time step $\delta$ is set as $0.1$ (see Remark~\ref{rem:det_omeghn}). In all simulations, the images are represented in a large enough box that contains tumour domain depicted therein well in its interior. The MATLAB code for NUM simulations can be found in the URL \href{https://github.com/gopikrishnancr/2D_tumour_growth_FEM_FVM}{\texttt{https://github.com/gopikrishnancr/2D\_tumour\_growth\_FEM\_FVM.}}

\subsection{Setting for NUM simulations (Set-NUM)}
\label{subsec:num}
We simulate the evolution of tumours starting with initial domains of the shapes as in Figures~\ref{fig:rand_c}--\ref{fig:rand_d}. In all the simulations, the dimension of the square $\Omega_\ell$ is $(-5,5)^2$. The final time is set at $T = 20$. The triangulations are as in Figures \ref{fig:tria_circle}--\ref{fig:tria_donut}.

In the simulations corresponding to Figure~\ref{fig:tumour_nsm}, we set $Q = 0.5$ and $\eta = 1$.

In Figure~\ref{fig:tumour_nsm}, we show the state of the variables: volume fraction, nutrient concentration, negative pressure, and the momentum -- defined as the product of the volume fraction and the cell velocity vector field --  at the time $T = 20$ from the top row to the bottom row, respectively. The columns from the left to the right depict the evolution of a tumour initially seeded with cells in the shape of a circle, bullet and semi-annulus, respectively. 

\subsection{Setting for NLM simulations (Set-NLM)}
\label{subsec:nlm}
In Set-NLM tests, we study the evolution of a tumour that was circular initially.  The dimension of the square $\Omega_\ell$ is $(-5,5)^2$ and the final time $T = 30$. We set $Q = 0.01$ and $\eta = 2$. It is worthwhile to notice that we keep $\eta$ to be the same inside and outside the tumour region for simplicity. However, in a more generic situation, $\eta$ will vary between the tumour region and external medium. In this set of experiments, volume fraction and nutrient concentration are solved in the entire spatial domain $\Omega_\ell$, while cell velocity and pressure are solved in $\Omega_h^n$ at each $t_n$.
\include{doc_figures}

\include{doc_figures_2}
We set the boundary values of the nutrient concentration $c$ as follows: $c = 0$ on $y = 5$ and $x = 5$, and $c = 1$ on $y = -5$ and $x = -5$. The initial nutrient concentration is given by $c_0(0,{\boldsymbol x}) = 0$.

In Figure~\ref{fig:tumour_nlm}, the columns from the left to the right show the state of the variables at time $T=10,\,20$ and $30$, respectively. The rows from the top to the bottom represent, volume fraction, nutrient concentration, negative pressure, and  cell momentum vector field, respectively.

\subsection{Discussion on numerical results}

\subsubsection{Set-NUM, effect of initial tumour shape}

Numerical experiments in subsections~\ref{subsec:num} and~\ref{subsec:nlm} substantiate the beneficial aspects of the discrete scheme (Definition~\ref{defn:discrete_scheme}) developed in Section~\ref{sec:n_scheme}. This scheme is able to simulate tumour geometries with arbitrary shapes (see Figure~\ref{fig:tumour_nsm}). Firstly, we considered a tumour with unit circular shaped initial geometry in Set-NUM and in this case, the initial volume fraction is uniform and symmetric about the origin. The nutrient concentration at the boundary of the tumour is unity throughout the simulation. Therefore, the tumour does not experience any unbalanced force that disturbs its symmetry and we expect radially symmetric growth. The numerical results in Figure~\ref{fig:vf_circle},~\ref{fig:ot_circle},~\ref{fig:pres_circle}, and~\ref{fig:vt_circle} confirm this argument. It is clear that the tumour is growing with radial symmetry as the volume fraction distribution in Figure~\ref{fig:vf_circle} indicates. However, such symmetry cannot be expected for the cases with asymmetric initial geometries. This is corroborated by the numerical experiments with the bullet shaped and semi-annular shaped initial geometry. In the case of a bullet shaped initial geometry, since much of the volume fraction is distributed along the $y$-axis rather than along the $x$-axis, a natural expectation is that the vertical dimension of the tumour is longer than the horizontal dimension, which the numerical simulations show. The asymmetric growth in the case of the tumour with semi-annular initial geometry arises in a different way. The convex side of the  tumour with apex at $x=1$ grows normally outwards, while the non-convex side grows into the semi-annular gap between $y = -0.5$ and $y = 0.5$, and $x = 0
$ and $x = 0.5$ (see Figure~\ref{fig:vf_donut} and~\ref{fig:vt_donut}). 
\begin{figure}[h!]
	\centering
	\begin{minipage}{0.2cm}
		\centering
		\rotatebox{90}{\textcolor{red}{\quad volume fraction}}
	\end{minipage}
	\begin{minipage}{14.5cm}
		\begin{subfigure}[b]{0.32\textwidth}
			\caption*{$T = 0$}
			\includegraphics[scale=0.9]{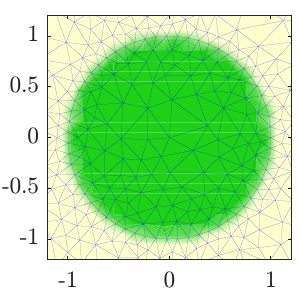}
			\caption{}
			\label{fig:vf_d_1}
		\end{subfigure}
		~ 
		\begin{subfigure}[b]{0.33\textwidth}
			\caption*{$T = 4$}
			\includegraphics[scale=0.9]{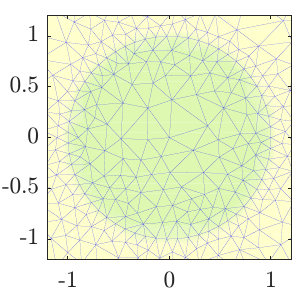}
			\caption{}
			\label{fig:vf_d_3}
		\end{subfigure}
		\begin{subfigure}[b]{0.32\textwidth}
			\caption*{$T = 8$}
			\includegraphics[scale=0.9]{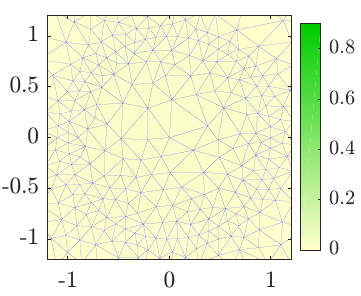}
			\caption{}
			\label{fig:vf_d_4}
		\end{subfigure}
	\end{minipage}
	\newline
	\begin{minipage}{0.2cm}
		\centering
		\rotatebox{90}{\textcolor{red}{\quad\quad nutrient concentration}}
	\end{minipage}
	\begin{minipage}{14.5cm}
		\begin{subfigure}[b]{0.32\textwidth}
			\includegraphics[scale=0.9]{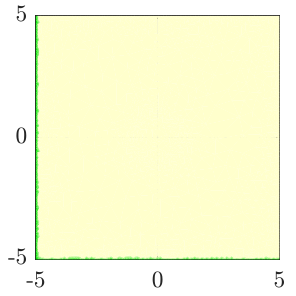}
			\caption{}
			\label{fig:ot_d_1}
		\end{subfigure}
		~ 
		\begin{subfigure}[b]{0.33\textwidth}
			\includegraphics[scale=0.9]{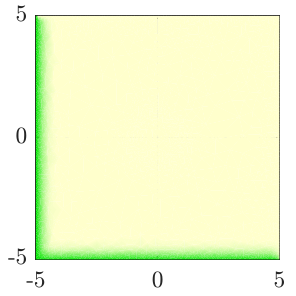}
			\caption{}
			\label{fig:ot_d_2}
		\end{subfigure}
		\begin{subfigure}[b]{0.32\textwidth}
			\includegraphics[scale=0.9]{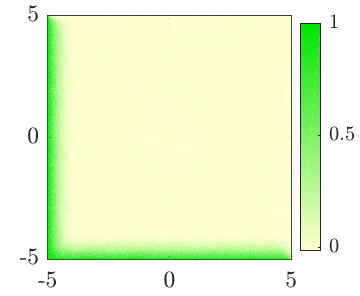}
			\caption{}
			\label{fig:ot_d_4}
		\end{subfigure}
	\end{minipage}
	\caption{The evolution of volume fraction and nutrient concentration with $\eta = 0.1$ and $Q = 0.01$ in NLM. Observe that the cells undergo necrosis before the nutrient can reach the tumour.}
	\label{fig:cnlm_var}
\end{figure}
As the tumour proliferates and expands, it becomes more difficult for the nutrient to diffuse into the interior region of tumour. The nutrient concentration distribution in Figures~\ref{fig:ot_circle},~\ref{fig:ot_bullet}, and~\ref{fig:ot_donut} show the decreasing value of concentration towards the interior of the tumour irrespective of the initial geometry. The depletion of nutrient level inside the tumour causes cell necrosis and as result, the extra-cellular fluid tends to fill the space generated. This is clearly reflected by the fact that the fluid pressure is more negative (see Figures~\ref{fig:pres_circle},~\ref{fig:pres_bullet}, and~\ref{fig:pres_donut}) towards the interior of the tumour and hence the fluid flow direction is from outside to inside. The cell velocity vector field shows the direction in which the cells are moving. When the initial geometry of the tumour is circular, the cells move in a radial direction with roughly equal magnitude (see Figure~\ref{fig:vt_circle}). However, in the case of asymmetric initial geometries the cell velocity vector field is also asymmetric (see Figures~\ref{fig:vt_bullet} and~\ref{fig:vt_donut}).

\subsubsection{Set-NLM, attraction towards oxygen source}

The simulations for the Set-NLM test give interesting results. It can be observed from the volume fraction at times $10,\,20$, and $30$ that the tumour grows towards the south-west corner. This affinity can be explained using the differential supply of the nutrient. The only source of the nutrient for the tumour comes from the left and bottom boundaries of the square $\Omega_\ell$. As Figures~\ref{fig:ot_nc_10},~\ref{fig:ot_nc_20} and~\ref{fig:ot_nc_30} show, the nutrient diffuses from the left and the bottom boundaries towards the tumour. The tumour starts to grow when this diffused nutrient reaches its vicinity. From Figure~\ref{fig:vfrac_nc_10}, we see that the tumour has not grown, until $T = 10$, the time at which the diffused nutrient just meets the tumour boundary. The tumour starts to grow after this time as observed from Figures~\ref{fig:vfrac_nc_20} and~\ref{fig:vfrac_nc_30}.
\begin{figure}[h!]
	\centering
	\hspace{-1.5cm}
	\begin{subfigure}[b]{0.44\textwidth}
		\includegraphics[scale=0.3]{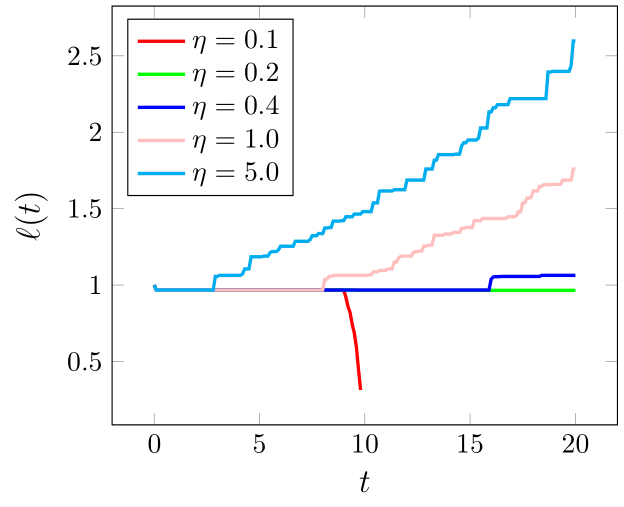}
		\caption{$Q = 0.01$}
		\label{fig:rad_dvt_var}
	\end{subfigure}
	~ 
	\begin{subfigure}[b]{0.44\textwidth}
		\includegraphics[scale=0.3]{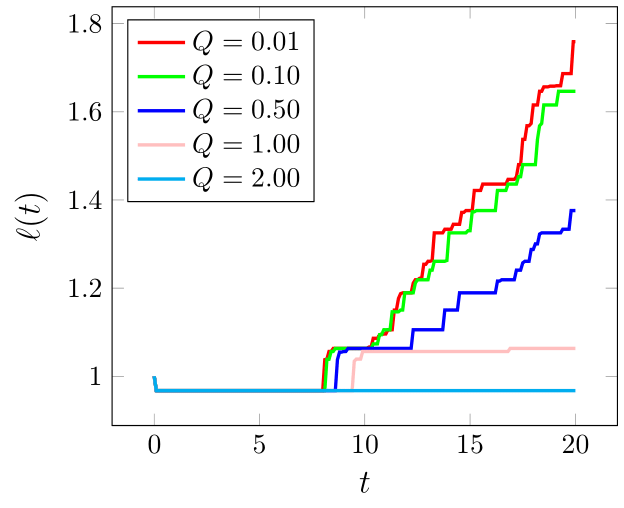}
		\caption{$\eta = 1.0$}
		\label{fig:rad_q_variation}
	\end{subfigure}
	\caption{Variation of the tumour radius, $\ell(t)$ with respect to the time for different values of $\eta$ and $Q$.}
\end{figure}
The numerical values of $Q$ and $\eta$ are crucial in determining the fate of the tumour. In fact, the diffusivity, $\eta$, which controls the ease of nutrient to diffuse into the tumour and the surrounding medium needs to be high enough so that the nutrient is able to reach the tumour vicinity before all the cells die. This situation occurs with numerical values $Q = 0.01$ and $\eta = 0.1$. Here, the low value of $\eta$ prevents the nutrient from reaching the tumour cells in adequate time (see Figures~\ref{fig:ot_d_1}-\ref{fig:ot_d_4}), and as a result the volume fraction of the tumour cells gradually decreases (see Figures~\ref{fig:vf_d_1}-\ref{fig:vf_d_4}). Moreover, this suggests that a higher value of $\eta$ facilitates faster tumour growth owing to faster diffusion of the nutrient, and is supported by the numerical results in Figure~\ref{fig:rad_dvt_var}. Here, the growth (set-NLM) of a tumour with circular initial geometry is studied, and we quantify the tumour size by the tumour radius, $\ell(t)$. Furthermore, we see that the tumour size decreases as $Q$ increases, indicated by Figure~\ref{fig:rad_q_variation}. We note that, broadly speaking, increasing $\eta$ and decreasing $Q$ have a similar effect in producing a larger tumour volume (see Figures~\ref{fig:rad_dvt_var}-\ref{fig:rad_q_variation}). In this way, identifiability issues may be encountered when estimating these two parameters from data that solely measures tumour size over time. However, supplementing with additional data on oxygen perfusion through cancer tissue (see, for example, \cite{grimes2016estimating}), we expect that both parameters could be estimated. 

\subsubsection{Handling topology changes of tumour}

Another notable feature of scheme is that it can simulate tumour growth starting from highly irregular initial geometries with multiple disconnected components. We consider growth of a tumour initially having three disconnected components with irregular boundaries. The irregularity of the initial tumour geometry is shown in Figure~\ref{fig:vfrac_irr_1}. The cell volume fraction at times $T = 0,\,5,\,10,\,20,\,20,\,30,$ and $40$ is plotted in Figure~\ref{fig:irregular_growth}. As the tumour grows the multiple components merge and the tumour continues to grow as a single entity. The numerical scheme is designed in such a way that intrinsic changes in the tumour geometry like the variation in the number of connected components is seamlessly dealt with and the numerical results in Figure~\ref{fig:irregular_growth} support this. It can be observed from Figure~\ref{fig:vfrac_irr_40} that a necrotic core of dead cells has developed owing to the nutrient starvation experienced at the tumour centre due to its large size.
The numerical scheme captures a broad spectrum of features as discussed previously for both symmetric and asymmetric initial geometries. A key factor that helps to achieve this is the implicit recovery of the boundary using the volume fraction. In the scheme it is not required to follow the movement of each point in the boundary, which may result in overlapping of edges and other similar complexities. Defining the interior of the tumour as the union of triangles with active cell volume fraction eliminates these issues, thereby making the numerical scheme versatile for a wide range of scenarios.

\subsubsection{Grid orientation effect}

It should be also noted that orientation of the triangulation has little effect in determining the tumour radius. The numerical experiments in Figure~\ref{fig:rotation_eff} illustrate this. In these simulations, three rotated versions (by angles $0$, $\pi/2$ and $\pi$) of a random triangulation are used for Set-NUM experiments, with an initial tumour in the form of a disk (this ensures that the rotated triangulations remain suitable for this initial shape, as detailed in Section \ref{sec:approximation}). The resulting volume fraction profiles remain mostly circular, with slight effects of the rotations but no change in the final tumour radius.

\begin{figure}
	\resizebox{0.9\textwidth}{!}{
		\begin{minipage}{0.2cm}
			\centering
			\rotatebox{90}{\textcolor{red}{\quad\quad triangulation}}
		\end{minipage}
		\begin{minipage}{14.5cm}
			\begin{subfigure}[b]{0.31\textwidth}
				\includegraphics[scale=1]{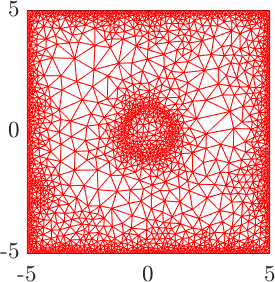}
				\caption{$\vartheta = 0$ rad}
				\label{fig:tria_rot1}
			\end{subfigure}
			~ 
			\begin{subfigure}[b]{0.33\textwidth}
				\includegraphics[scale=1]{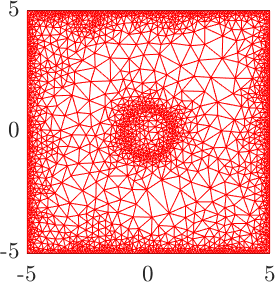}
				\caption{$\vartheta = \pi/2$ rad}
				\label{fig:tria_rot2}
			\end{subfigure}
			\begin{subfigure}[b]{0.32\textwidth}
				\includegraphics[scale=1]{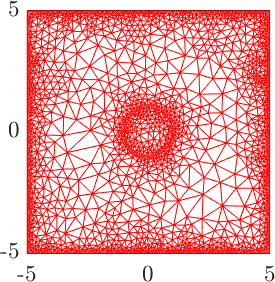}
				\caption{$\vartheta = \pi$ rad}
				\label{fig:tria_rot3}
			\end{subfigure}
		\end{minipage}
	}
	\newline
	\resizebox{0.9\textwidth}{!}{
		\centering
		\begin{minipage}{0.2cm}
			\centering
			\rotatebox{90}{\textcolor{red}{\quad volume fraction}}
		\end{minipage}
		\begin{minipage}{14.5cm}
			\begin{subfigure}[b]{0.31\textwidth}
				\includegraphics[scale=1]{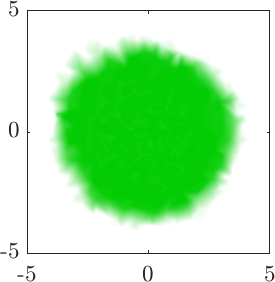}
				\caption{}
				\label{fig:vfrac_rot1}
			\end{subfigure}
			~ 
			\begin{subfigure}[b]{0.33\textwidth}
				\includegraphics[scale=1]{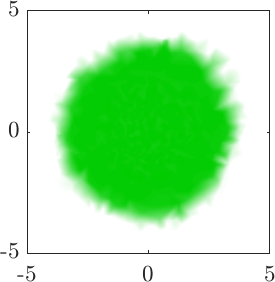}
				\caption{}
				\label{fig:vfrac_rot2}
			\end{subfigure}
			\begin{subfigure}[b]{0.32\textwidth}
				\includegraphics[scale=1]{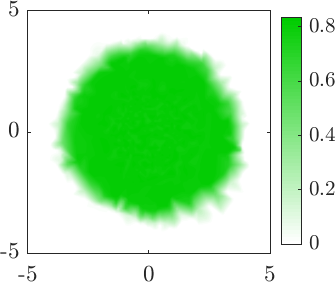}
				\caption{}
				\label{fig:vfrac_rot3}
			\end{subfigure}
		\end{minipage}
	}
	\caption{Effect of orientation of the triangulation on tumour radius. In Figures~\ref{fig:tria_rot1}--\ref{fig:tria_rot3} the triangulation is rotated anticlockwise by the angles $\vartheta = 0,\,\pi/2$ and $\pi$ radians. The corresponding volume fraction profile at $T = 20$ with temporal discretisation factor $\delta  = 0.1$ is provided in Figures~\ref{fig:vfrac_rot1}--\ref{fig:vfrac_rot3}.}
	\label{fig:rotation_eff}
\end{figure}

\subsubsection{Using structured meshes}

The use of a random Delaunay mesh is critical in obtaining good solutions that have minimal mesh-locking. We present the evolution of the volume fraction of a tumour starting with a circular initial geometry, simulated using  structured triangulations with 1024, 4096, and 16,384 triangles in Figures~\ref{fig:v_5_1}--~\ref{fig:v_5_3}, Figures~\ref{fig:v_6_1}--~\ref{fig:v_6_3}, and Figures~\ref{fig:v_7_1}--~\ref{fig:v_7_3}, respectively. The final time is set as $T = 20$, and the time step is $\delta = 0.1$. The initial geometry is circular (see Figure~\ref{fig:v_7_1}). As the triangulations become more refined, it can be observed that the tumour becomes more radially symmetrical. This observation indicates the convergence of the discrete solutions to the radially symmetric solution as the spatial discretisation factor approaches zero. However, the tumour also becomes more squarish as time increases, as shown in Figure~\ref{fig:vf_st_ev}, showing that, for a long time, an extremely fine structure triangulation would have to be used to obtain a reasonable
solution. Such refinement would come at a great cost, whereas the use of a random mesh (with adaptation only to the initial shape) provides suitable solutions with relatively few triangles. 

\begin{figure}
	\centering
	\begin{subfigure}[b]{0.31\textwidth}
		\includegraphics[scale=0.36]{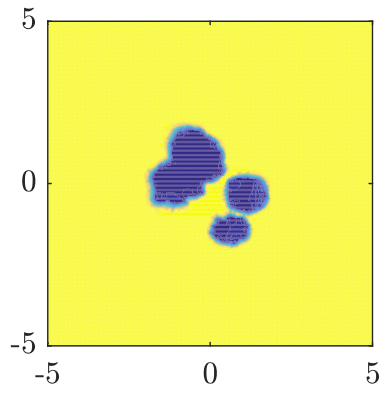}
		\caption{$ T = 0$}
		\label{fig:vfrac_irr_1}
	\end{subfigure}
	~ 
	\begin{subfigure}[b]{0.33\textwidth}
		\includegraphics[scale=0.36]{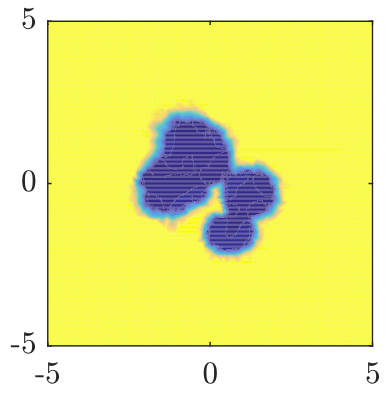}
		\caption{$T = 5$}
		\label{fig:vfrac_irr_5}
	\end{subfigure}
	\begin{subfigure}[b]{0.32\textwidth}
		\includegraphics[scale=0.36]{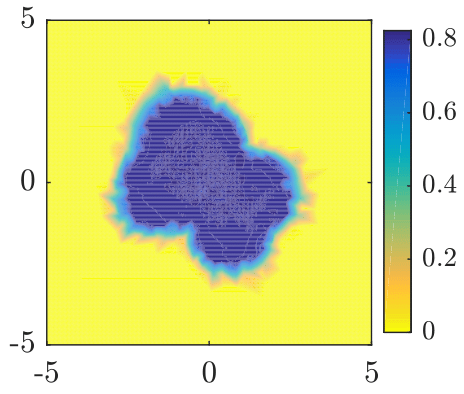}
		\caption{$T = 10$}
		\label{fig:vfrac_irr_10}
	\end{subfigure}
	\begin{subfigure}[b]{0.31\textwidth}
		\includegraphics[scale=0.36]{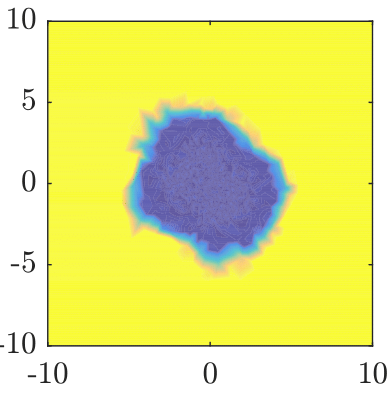}
		\caption{$ T = 20$}
		\label{fig:vfrac_irr_20}
	\end{subfigure}
	~ 
	\begin{subfigure}[b]{0.33\textwidth}
		\includegraphics[scale=0.36]{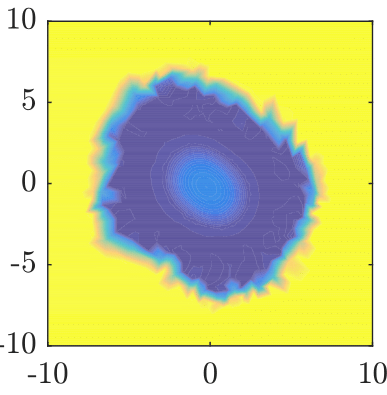}
		\caption{$T = 30$}
		\label{fig:vfrac_irr_30}
	\end{subfigure}
	\begin{subfigure}[b]{0.32\textwidth}
		\includegraphics[scale=0.36]{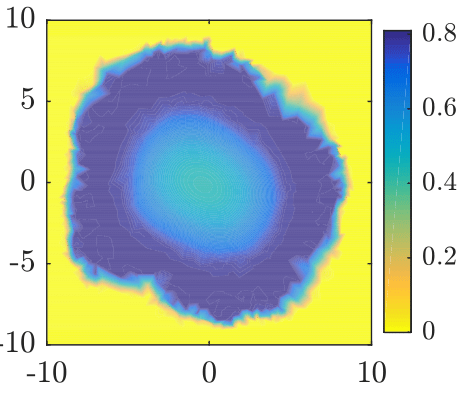}
		\caption{$T = 40$}
		\label{fig:vfrac_irr_40}
	\end{subfigure}
	\caption{Stages of cell volume fraction for tumour growth (NUM) with an irregular initial shape having multiple initial components. }
	\label{fig:irregular_growth}
\end{figure}

\begin{figure}
	\resizebox{0.95\textwidth}{!}{
		\centering
		\begin{subfigure}[b]{0.32\textwidth}
			\caption*{\quad\quad $T = 0$}
			\includegraphics[scale=1]{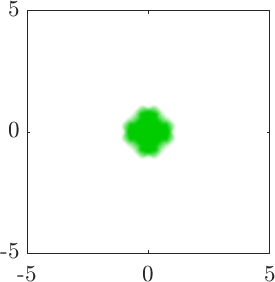}
			\caption{}
			\label{fig:v_5_1}
		\end{subfigure}
		\begin{subfigure}[b]{0.32\textwidth}
			\caption*{\quad\quad $T = 10$}
			\includegraphics[scale=1]{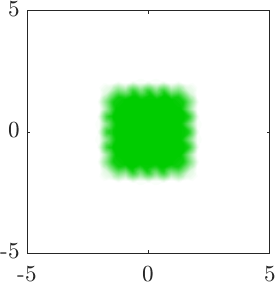}
			\caption{}
			\label{fig:v_5_2}
		\end{subfigure}
		\begin{subfigure}[b]{0.32\textwidth}
			\caption*{\quad\quad $T = 20$}
			\includegraphics[scale=1]{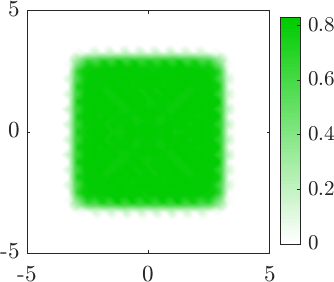}
			\caption{}
			\label{fig:v_5_3}
		\end{subfigure}
	}
	\newline
	\resizebox{0.95\textwidth}{!}{
		\centering
		\begin{subfigure}[b]{0.32\textwidth}
			\includegraphics[scale=1]{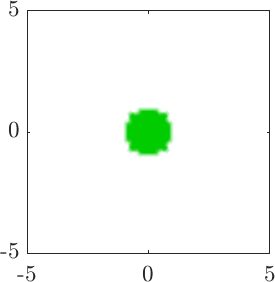}
			\caption{}
			\label{fig:v_6_1}
		\end{subfigure}
		\begin{subfigure}[b]{0.32\textwidth}
			\includegraphics[scale=1]{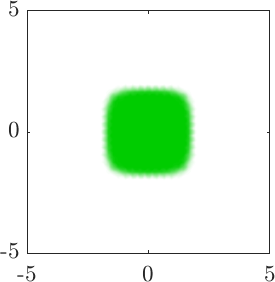}
			\caption{}
			\label{fig:v_6_2}
		\end{subfigure}
		\begin{subfigure}[b]{0.32\textwidth}
			\includegraphics[scale=1]{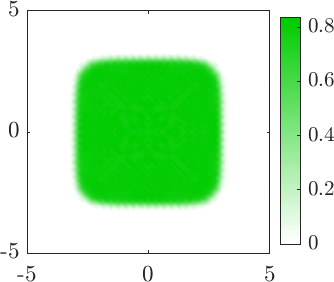}
			\caption{}
			\label{fig:v_6_3}
		\end{subfigure}
	}
	\resizebox{0.95\textwidth}{!}{
		\centering
		\begin{subfigure}[b]{0.32\textwidth}
			\includegraphics[scale=1]{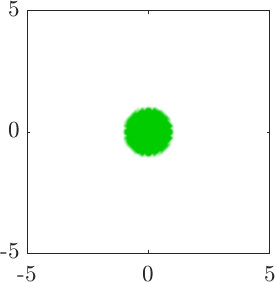}
			\caption{}
			\label{fig:v_7_1}
		\end{subfigure}
		\begin{subfigure}[b]{0.32\textwidth}
			\includegraphics[scale=1]{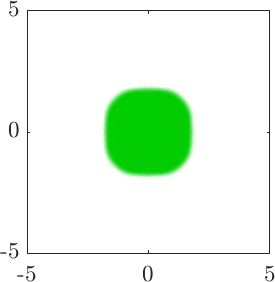}
			\caption{}
			\label{fig:v_7_2}
		\end{subfigure}
		\begin{subfigure}[b]{0.32\textwidth}
			\includegraphics[scale=1]{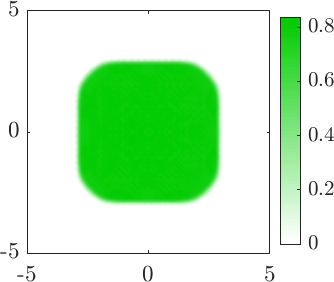}
			\caption{}
			\label{fig:v_7_3}
		\end{subfigure}
	}
	\caption{Evolution of volume fraction with respect to time on structured triangulation. The initial domain is a circle centred at origin with unit radius. Figures~\ref{fig:v_5_1}--~\ref{fig:v_5_3} are computed using the triangulation in Figure~\ref{fig:trian_st_5}, Figures~\ref{fig:v_6_1}--~\ref{fig:v_6_3} are computed using the triangulation in Figure~\ref{fig:trian_st_6}, and Figures~\ref{fig:v_7_1}--~\ref{fig:v_7_3} are computed using the triangulation in Figure~\ref{fig:trian_st_7}. }
	\label{fig:vf_st_ev}
\end{figure}

\subsubsection{Assessment of convergence}

\begin{figure}
	\resizebox{0.9\textwidth}{!}{
		\begin{minipage}{0.2cm}
			\centering
			\rotatebox{90}{\textcolor{red}{\quad\quad triangulation}}
		\end{minipage}
		\begin{minipage}{14.5cm}
			\begin{subfigure}[b]{0.31\textwidth}
				\includegraphics[scale=1]{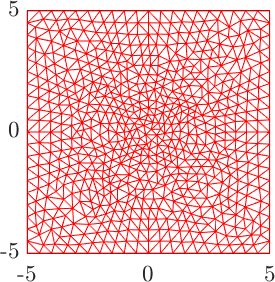}
				\caption{}
				\label{fig:conv_tria_ref1}
			\end{subfigure}
			~ 
			\begin{subfigure}[b]{0.33\textwidth}
				\includegraphics[scale=1]{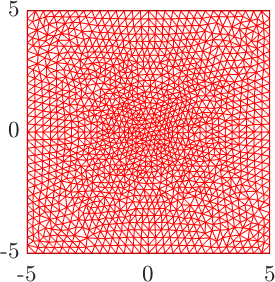}
				\caption{}
				\label{fig:conv_tria_ref2}
			\end{subfigure}
			\begin{subfigure}[b]{0.32\textwidth}
				\includegraphics[scale=1]{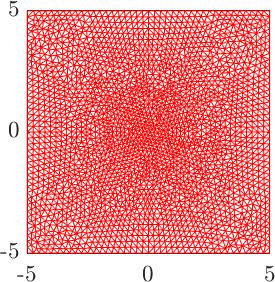}
				\caption{}
				\label{fig:conv_tria_ref3}
			\end{subfigure}
		\end{minipage}
	}
	\newline
	\resizebox{0.9\textwidth}{!}{
		\centering
		\begin{minipage}{0.2cm}
			\centering
			\rotatebox{90}{\textcolor{red}{\quad volume fraction}}
		\end{minipage}
		\begin{minipage}{14.5cm}
			\begin{subfigure}[b]{0.31\textwidth}
				\includegraphics[scale=1]{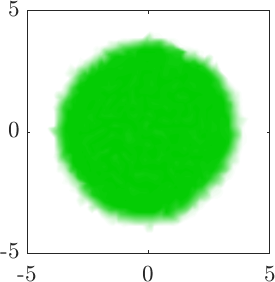}
				\caption{}
				\label{fig:conv_vfrac_ref1}
			\end{subfigure}
			~ 
			\begin{subfigure}[b]{0.33\textwidth}
				\includegraphics[scale=1]{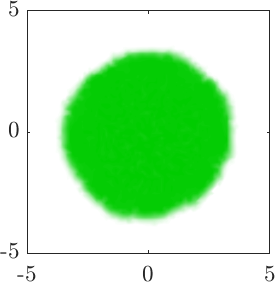}
				\caption{}
				\label{fig:conv_vfrac_ref2}
			\end{subfigure}
			\begin{subfigure}[b]{0.32\textwidth}
				\includegraphics[scale=1]{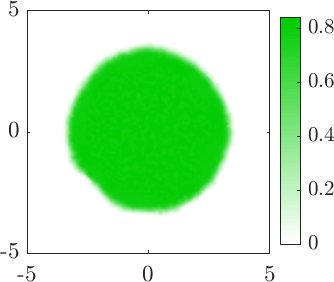}
				\caption{}
				\label{fig:conv_vfrac_ref3}
			\end{subfigure}
		\end{minipage}
	}
	\caption{Convergence of cell volume fraction for Set-NLM with respect to the spatial discretisation factor. The triangulations in Figures~\ref{fig:conv_tria_ref1},~~\ref{fig:conv_tria_ref2}, and~~\ref{fig:conv_tria_ref3} contains 1248, 2084, and 4996 triangles. The volume fractions for Set-NLM are computed at the time $T = 20$.}
	\label{fig:random_conv}
\end{figure}

The convergence of the scheme, as the grid size is reduced, is clearly observable in the case of random triangulations; see Figure~\ref{fig:random_conv}. However, this convergence requires uniform refinements of the mesh, because it depends on both on a Courant--Friedrichs--Lewy (CFL) and on an inverse CFL relation, as demonstrated in~\cite{DNR19}. These conditions take the form
\begin{equation}
\underbrace{\mathrm{C}_{\textsc{icfl}} \le \max_{0 \le n \le N}\sup_{\Omega_h^n} ||\boldsymbol{u}_h^n||_2 \dfrac{\delta}{a_{\mathrm{max}}}}_{\text{inverse CFL condition}} \le \overbrace{\max_{0 \le n \le N}\sup_{\Omega_h^n} ||\boldsymbol{u}_h^n||_2 \dfrac{\delta}{a_{\mathrm{min}}}  \le \mathrm{C}_{\textsc{cfl}}}^{\text{CFL condition}},
\label{eqn:cfl_cond}
\end{equation}
where $\mathrm{C}_{\textsc{icfl}}$ and $\mathrm{C}_{\textsc{cfl}}$ are positive constants, $a_{\mathrm{max}} = \max_{j}a_{j}$, $a_{\mathrm{min}} = \min_{j}a_{j}$, $||{\cdot}||_2$ is the Euclidean norm; recall that $a_j$ is the area of triangle $j$. The temporal discretisation factor $\delta$ is fixed by the smallest triangle through the CFL condition~\eqref{eqn:cfl_cond}. With this $\delta$,  at each time step the diffusion of tumour cells inside the larger triangles would not be sufficient to create a volume fraction $\alpha_h^n$ larger than the threshold, and the tumour would not expand. Such a situation is avoided by the inverse CFL condition~\ref{eqn:cfl_cond}, which ensures a lower bound on numerical diffusion on large triangles also. Nevertheless, the CFL and inverse CFL condition together restrict the possible choices of temporal discretisation factor. Since Ruppert's algorithm performs a fine refinement on  triangles near the boundaries of the initial domain and bounding box, and a relatively coarser refinement on the triangles in between these two boundaries, it leads to a refined triangulation with considerable difference in the sizes of triangles within. Therefore, in the case of very fine refinements, it is better to consider a structured triangulation well adapted to the initial condition, and then perturb the vertices of triangles randomly to remove the mesh-locking effect (see Figures~\ref{fig:conv_tria_ref1}--\ref{fig:conv_tria_ref3}). It can be observed from Figures~\ref{fig:conv_vfrac_ref1}--~\ref{fig:conv_vfrac_ref3} that the volume fractions are indeed converging with mesh refinement.

\begin{figure}[h!]
	\centering
	\begin{subfigure}[b]{0.44\textwidth}
		\includegraphics[scale=1]{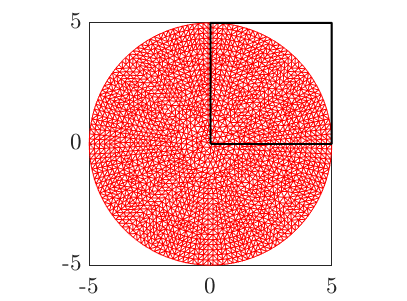}
		\caption{}
		\label{fig:trian_circ}
	\end{subfigure}
	~ 
	\begin{subfigure}[b]{0.44\textwidth}
		\includegraphics[scale=1]{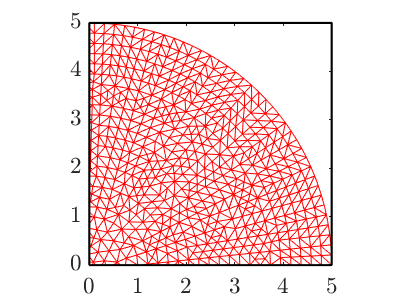}
		\caption{}
		\label{fig:trian_circ_zoom}
	\end{subfigure}
	\caption{Radially aligned triangulation - Figure~\ref{fig:trian_circ} shows the triangulation on the domain $\Omega_{\ell} = (-5,5)^2$ and Figure~\ref{fig:trian_circ_zoom} shows an enlarged view of the first quadrant.}
\end{figure}

\begin{figure}[htp]
	\resizebox{0.97\textwidth}{!}{
		\centering
		\begin{subfigure}[b]{0.33\textwidth}
			\includegraphics[scale=1]{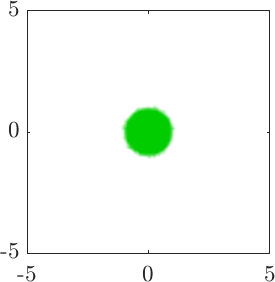}
			\caption{T = 0}
			\label{fig:vfrac_c_1}
		\end{subfigure}
		\begin{subfigure}[b]{0.33\textwidth}
			\includegraphics[scale=1]{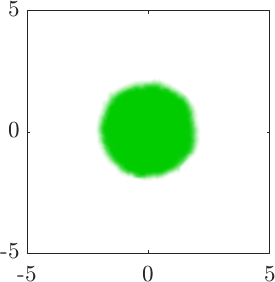}
			\caption{T = 10}
			\label{fig:vfrac_c_2}
		\end{subfigure}
		\begin{subfigure}[b]{0.33\textwidth}
			\includegraphics[scale=1]{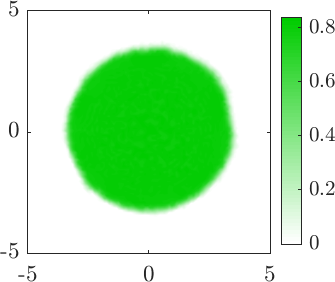}
			\caption{T = 20}
			\label{fig:vfrac_c_3}
		\end{subfigure}
	}
	\caption{Evolution of volume fraction obtained from Set--NUM experiment on the radially aligned triangulation in Figure~\ref{fig:trian_circ}.}
\end{figure}

Mesh locking and loss of radial symmetry in the case of structured triangulations is not due to the procedure using a threshold value to capture the boundary of a tumour. Instead, this is a classical problem associated with the nature of triangulations and finite volume schemes (see subsection~\ref{sec:mesh_lcoking} also). If the symmetry of a discrete solution is known \emph{a priori} and we use a triangulation that respects this symmetry, then the discrete scheme in Definition~\ref{defn:discrete_scheme} preserves this symmetry. For instance, consider the evolution of a tumour with an initial geometry of a unit circle centred at the origin. Since the tumour is expected to evolve with a radial symmetry, we use a triangulation wherein the triangles are aligned with concentric circles centred at the origin (see Figure~\ref{fig:trian_circ}). In this case, it can be observed from Figures~\ref{fig:vfrac_c_1}--\ref{fig:vfrac_c_3} that the discrete volume fraction remains radially symmetrical. However, this method cannot be used in the case of initial geometries like the bullet or semi--annular shape since the symmetry properties of discrete solutions are not known \emph{a priori}. In such cases, the most economically viable choice is to resort to a random triangulation. 

\subsubsection{Influence of threshold value}

\begin{figure}[htp]
	\centering
	\begin{subfigure}[t]{0.48\textwidth}
		\centering
		\includegraphics[scale=0.9]{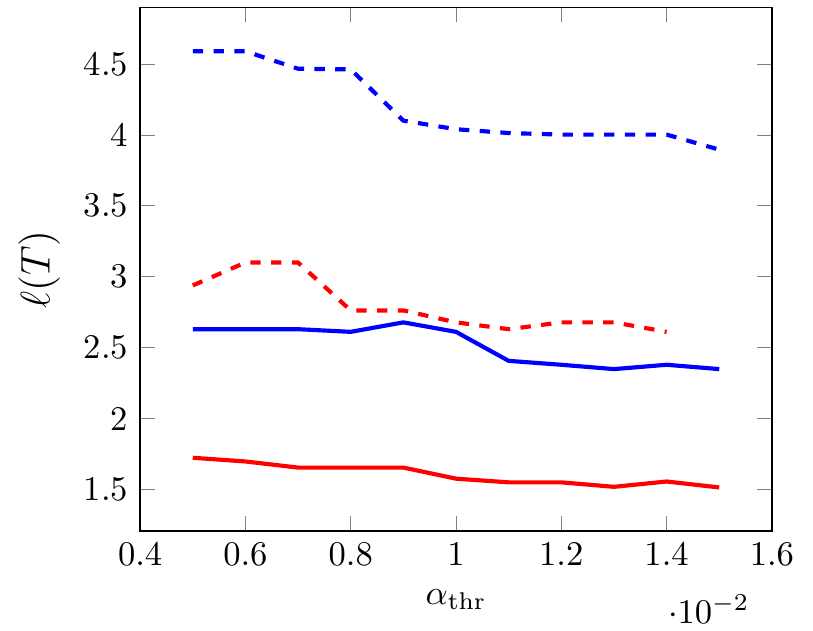}
		\caption{Set--NUM: $Q = 5$ -- solid lines, $Q = 0.5$ -- dashed lines, $\eta = 1$ -- blue lines, and $\eta = 0.1$ -- red lines.}
	\end{subfigure}
	\hspace{0.4cm}
	\begin{subfigure}[t]{0.48\textwidth}
		\centering
		\includegraphics[scale=0.9]{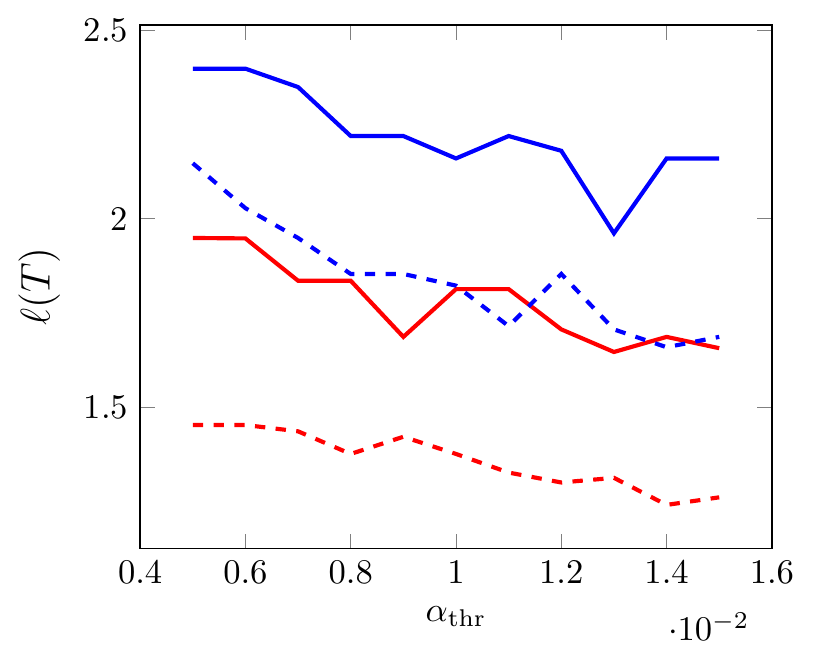}
		\caption{Set--NLM: $Q = 0.01$ -- solid lines, $Q = 0.5$ -- dashed lines, $\eta = 1$ -- blue lines, and $\eta =2$ -- red lines.}
	\end{subfigure}    
	\caption{Dependence of $\ell(T)$, where $T = 20$ on $\athr$.}
	\label{fig:thr_rad}
\end{figure}

The choice of threshold value, $\athr$,  influences the evolution of the tumour radius and hence, by extension, the other variables. We cannot choose the threshold value to be too large or too small. Such a choice will incur a cascading array of high errors on the tumour radius and other variables as the time increases. A very small threshold value implies that the volume fraction is too small on triangles closer to the boundary, thus forcing the velocity--pressure system to be singular.  The variation of tumour radius at the time $T = 20$ with respect to the threshold value over the range $[0.005,0.015]$ for Set--NUM and Set--NLM experiments is provided in Figure~\ref{fig:thr_rad}. The radius varies by a maximum of about 15\% for Set-NUM and 20\% for Set-NLM as the threshold value varies from $0.005$ to $0.015$. Therefore, deviation in the tumour radius with respect to the threshold value is present. But, with a proper choice of the threshold value, it is possible to minimise the error in the tumour radius from the exact value~\cite{Remesan2019}. Moreover, one of the main motivations for simulating cancer growth is perhaps not to get an extremely accurate representation of the tumour radius, but more to study the effect of drugs; in this situation, the simulation of the current model would serve as a baseline, to be compared with simulations obtained with a model including said drug effect, and run using the same threshold value.

\section{Conclusions}
\label{sec:conclusion}
In this paper, a mathematically well-defined model is developed which can replicate the evolution of an avascular tumour that grows from a variety of initial geometries. The equivalent formulation in Section~\ref{sec:weak_equiv} and Theorem~\ref{thm:eq_thm_1} yield a framework to design a numerical scheme that does not require explicit tracking of the time-dependent boundary associated with the tumour. The tumour domain is recovered as the union of all triangles in which the volume fraction of the tumour is greater than a fixed threshold value. While implementing the scheme, a multitude of factors, like the nature of triangulation and the threshold value need to be taken into account. For instance, we illustrate by an example the mesh-locking effect associated with the use of structured triangulations and the advantage of using a random triangulation. The numerical results for both NUM and NLM models support the heuristic expectations and results from previous literature~\cite{breward_2002,hubbard}. The tests also illustrate the nutrient dependent growth of the tumour as in Figure~\ref{fig:tumour_nlm}. In addition to this, the numerical scheme seamlessly deals with the complex tumour geometries in Figure~\ref{fig:irregular_growth}, including initially disconnected tumour groups that merge later on. The numerical results justify the ability of the scheme to take care of different irregular tumour geometries and topological structures, which in turn shows its practical applicability in simulating tumour growth from real-time clinical data.  As such, the work presented here could be extended to quantify the effect of drug treatment on an evolving tumour. 

\subsubsection*{Acknowledgement}
The authors are grateful to Prof. Neela Nataraj, Indian Institute of Technology Bombay, India for the valuable suggestions and help. The authors are grateful to Dr. Laura Bray, Queensland University of Technology, Australia and Ms Berline Murekatete, Queensland University of Technology, Australia for helpful discussions and providing image data for the irregular tumour depicted in Figure~\ref{fig:vfrac_irr_1}.

\subsubsection*{Data availability statement}
The datasets -- specifically, MATLAB code for NUM simulations -- generated during and/or analysed during the current study are available in the GitHub repository,\\ \href{https://github.com/gopikrishnancr/2D_tumour_growth_FEM_FVM}{\texttt{https://github.com/gopikrishnancr/2D\_tumour\_growth\_FEM\_FVM.}}

\bibliographystyle{plain}
\bibliography{2d_ref}

\appendix
\section*{Appendix}
\label{sec:appen}
\renewcommand{\thesubsection}{\Alph{subsection}}
\subsection{Some classical definitions and results}
We recall two classical results  used in this article. 
\label{classical}
\begin{enumerate}[label= \alph*.,ref=\ref{classical}.(\alph*)]
	\item \label{korns} \textbf{Theorem (Korn's second inequality). \cite[Theorem 3.78]{alexander}.}
	If $\Omega \subset \mathbb{R}^d$, where $d =2,3$ is a domain, then there exists a positive constant $\mathscr{C}_K$ such that, for every $\boldsymbol{v} \in \mathbf{H}^1_d(\Omega)$,
	\begin{equation}
	\mathscr{C}_K ||{\boldsymbol v}||_{1,\Omega} \le ||\nabla_s{\boldsymbol v}||_{0,\Omega} + ||{\boldsymbol v}||_{0,\Omega}.
	\end{equation}
	\item \label{tartar} \textbf{Lemma (Petree-Tartar). \cite[Lemma A.38]{alexander}.}
	If $X,\,Y,\,$ and $Z$ are Banach spaces,  $A : X \rightarrow Y$ is an injective operator, $T : X \rightarrow Z$ is a compact operator, and there exists a positive constant $\mathscr{C}_{1}$ such that $\mathscr{C}_{1} ||x||_X \le ||Ax||_Y + ||Tx||_Z$, then there exists a positive constant $\mathscr{C}_{PT}$ such that $\mathscr{C}_{PT} ||x||_X \le ||Ax||_Y$.
	\item \textbf{Definition (Bounded variation).}\label{def:bv} By the space $BV(A)$, where $A \subset \mathbb{R}^d$ is an open set we mean the collection of all functions $u : A \rightarrow \mathbb{R}$ such that $||u||_{BV} < \infty$, where 
	\begin{equation}
	||u||_{BV} := \sup\left\{\int_{A} u\,\mathrm{div}(\varphi)\,\vx : \varphi \in \mathscr{C}_c^1(A;\mathbb{R}^d), ||\varphi||_{L^\infty(A)} \le 1 \right\}.
	\end{equation}
\end{enumerate}
\end{document}

%% file: doc_figures.tex
\begin{figure}
\resizebox{0.85\textwidth}{!}{
	\centering
	\begin{minipage}{0.2cm}
		\centering
		\rotatebox{90}{\textcolor{red}{\quad volume fraction}}
	\end{minipage}
\begin{minipage}{14.5cm}
	\begin{subfigure}[b]{0.31\textwidth}
	\caption*{\quad\quad circular shaped}
	\includegraphics[scale=1]{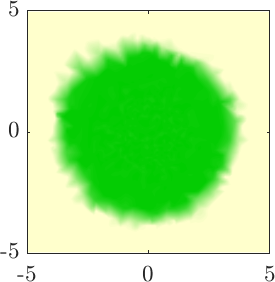}
	\caption{}
	\label{fig:vf_circle}
\end{subfigure}
~ 
\begin{subfigure}[b]{0.33\textwidth}
	\caption*{\quad\quad bullet shaped }
	\includegraphics[scale=1]{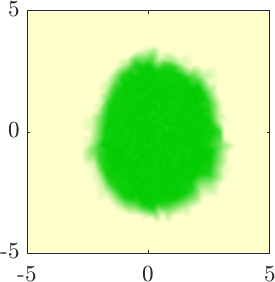}
	\caption{}
	\label{fig:vf_bullet}
\end{subfigure}
\begin{subfigure}[b]{0.32\textwidth}
	\caption*{\quad\quad semi-annular shaped}
	\includegraphics[scale=1]{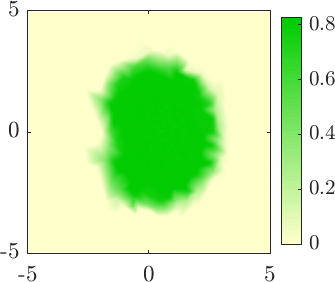}
	\caption{}
	\label{fig:vf_donut}
\end{subfigure}
\end{minipage}
}
\newline
\resizebox{0.85\textwidth}{!}{
	\begin{minipage}{0.2cm}
			\centering
	\rotatebox{90}{\textcolor{red}{\quad\quad nutrient concentration}}
\end{minipage}
\begin{minipage}{14.5cm}
\begin{subfigure}[b]{0.31\textwidth}
	\includegraphics[scale=1]{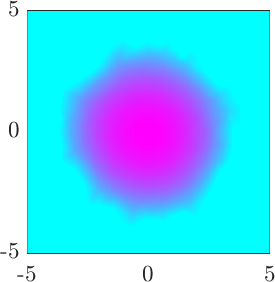}
	\caption{}
	\label{fig:ot_circle}
\end{subfigure}
~ 
\begin{subfigure}[b]{0.33\textwidth}
	\includegraphics[scale=1]{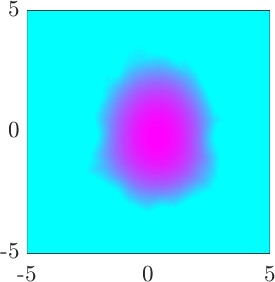}
	\caption{}
	\label{fig:ot_bullet}
\end{subfigure}
\begin{subfigure}[b]{0.32\textwidth}
	\includegraphics[scale=1]{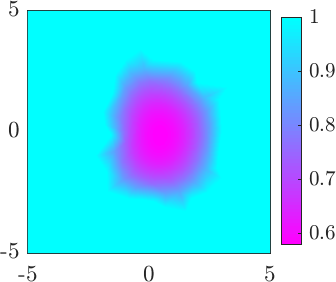}
	\caption{}
	\label{fig:ot_donut}
\end{subfigure}
\end{minipage}
}
\newline
\resizebox{0.85\textwidth}{!}{
	\begin{minipage}{0.2cm}
			\centering
	\rotatebox{90}{\textcolor{red}{\quad\quad negative pressure}}
\end{minipage}
\begin{minipage}{14.5cm}
\begin{subfigure}[b]{0.31\textwidth}
	\includegraphics[scale=1]{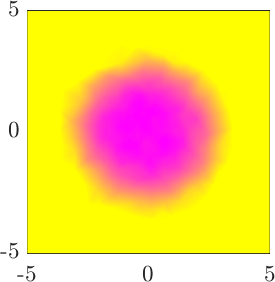}
	\caption{}
	\label{fig:pres_circle}
\end{subfigure}
~ 
\begin{subfigure}[b]{0.33\textwidth}
	\includegraphics[scale=1]{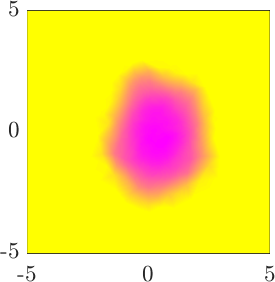}
	\caption{}
	\label{fig:pres_bullet}
\end{subfigure}
\begin{subfigure}[b]{0.32\textwidth}
	\includegraphics[scale=1]{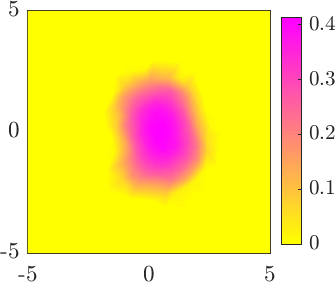}
	\caption{}
	\label{fig:pres_donut}
\end{subfigure}
\end{minipage}
}
\newline
\resizebox{0.85\textwidth}{!}{
	\begin{minipage}{0.2cm}
			\centering
	\rotatebox{90}{\textcolor{red}{\quad\quad momentum}}
\end{minipage}
\begin{minipage}{14.5cm}
\begin{subfigure}[b]{0.31\textwidth}
	\includegraphics[scale=1]{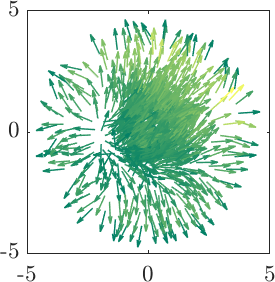}
	\caption{}
	\label{fig:vt_circle}
\end{subfigure}
~ 
\begin{subfigure}[b]{0.33\textwidth}
	\includegraphics[scale=1]{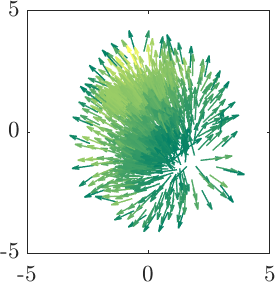}
	\caption{}
	\label{fig:vt_bullet}
\end{subfigure} 
\begin{subfigure}[b]{0.32\textwidth}
	\includegraphics[scale=1]{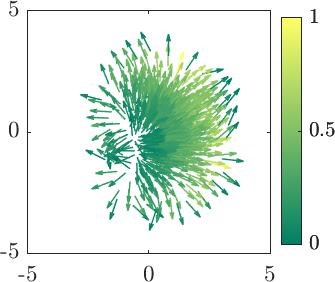}
	\caption{}
	\label{fig:vt_donut}
\end{subfigure}
\end{minipage}
}
\caption{Set-NUM: Rows one to four illustrate the volume fraction, nutrient concentration, negative pressure, and cell momentum at $T = 20$, respectively. The variables in columns one to three correspond to an initial domain, $\Omega(0)$, in the shape of a circle, bullet, and semi-annulus, respectively.}
	\label{fig:tumour_nsm}
\end{figure}

%% file: doc_figures_2.tex
\begin{figure}[h!]
\resizebox{0.85\textwidth}{!}{
	\centering
	\begin{minipage}{0.2cm}
		\centering
		\rotatebox{90}{\textcolor{red}{\quad\quad volume fraction}}
	\end{minipage}
	\begin{minipage}{14.5cm}
		\begin{subfigure}[b]{0.31\textwidth}
			\caption*{\quad\quad $T = 10$}
			\includegraphics[scale=1]{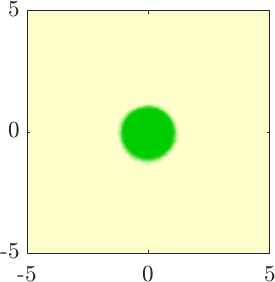}
			\caption{}
			\label{fig:vfrac_nc_10}
		\end{subfigure}
		~ 
		\begin{subfigure}[b]{0.33\textwidth}
			\caption*{\quad\quad $T = 20$}
			\includegraphics[scale=1]{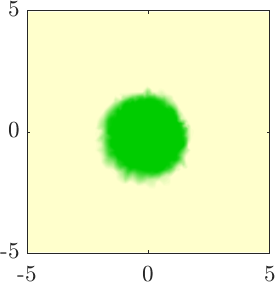}
			\caption{}
			\label{fig:vfrac_nc_20}
		\end{subfigure}
		\begin{subfigure}[b]{0.32\textwidth}
			\caption*{\quad\quad $T = 30$}
			\includegraphics[scale=1]{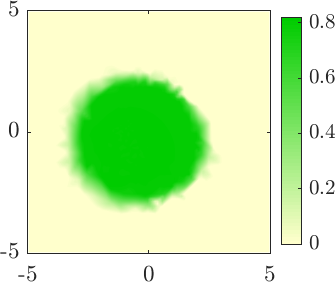}
			\caption{}
			\label{fig:vfrac_nc_30}
		\end{subfigure}
	\end{minipage}
}
	\newline
\resizebox{0.85\textwidth}{!}{
	\begin{minipage}{0.2cm}
		\centering
		\rotatebox{90}{\textcolor{red}{\quad\quad nutrient concentration}}
	\end{minipage}
	\begin{minipage}{14.5cm}
		\begin{subfigure}[b]{0.31\textwidth}
			\includegraphics[scale=1]{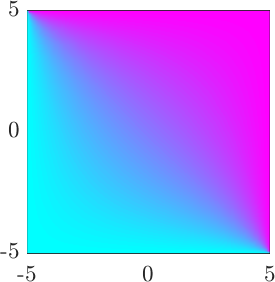}
			\caption{}
			\label{fig:ot_nc_10}
		\end{subfigure}
		~ 
		\begin{subfigure}[b]{0.33\textwidth}
			\includegraphics[scale=1]{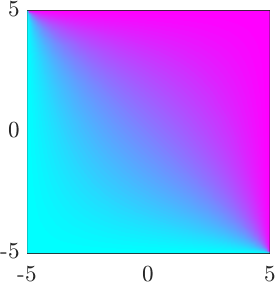}
			\caption{}
			\label{fig:ot_nc_20}
		\end{subfigure}
		\begin{subfigure}[b]{0.32\textwidth}
			\includegraphics[scale=1]{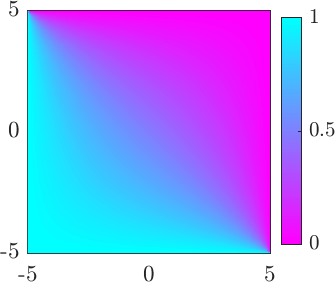}
			\caption{}
			\label{fig:ot_nc_30}
		\end{subfigure}
	\end{minipage}
}
	\newline
\resizebox{0.85\textwidth}{!}{
	\begin{minipage}{0.2cm}
		\centering
		\rotatebox{90}{\textcolor{red}{\quad\quad negative pressure}}
	\end{minipage}
	\begin{minipage}{14.5cm}
		\begin{subfigure}[b]{0.31\textwidth}
			\includegraphics[scale=1]{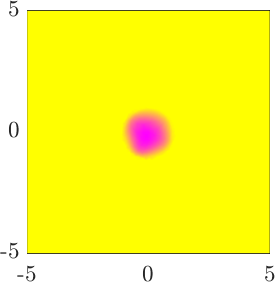}
			\caption{}
			\label{fig:pres_nc_10}
		\end{subfigure}
		~ 
		\begin{subfigure}[b]{0.33\textwidth}
			\includegraphics[scale=1]{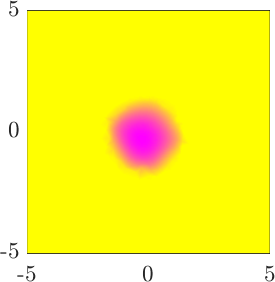}
			\caption{}
			\label{fig:pres_nc_20}
		\end{subfigure}
		\begin{subfigure}[b]{0.32\textwidth}
			\includegraphics[scale=1]{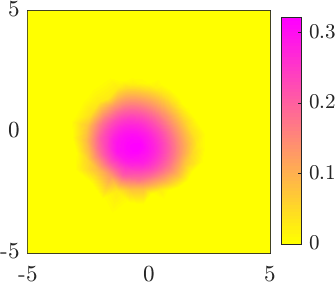}
			\caption{}
			\label{fig:pres_nc_30}
		\end{subfigure}
	\end{minipage}
}
	\newline
\resizebox{0.85\textwidth}{!}{
	\begin{minipage}{0.2cm}
		\centering
		\rotatebox{90}{\textcolor{red}{\quad\quad momentum}}
	\end{minipage}
	\begin{minipage}{14.5cm}
		\begin{subfigure}[b]{0.31\textwidth}
			\includegraphics[scale=1]{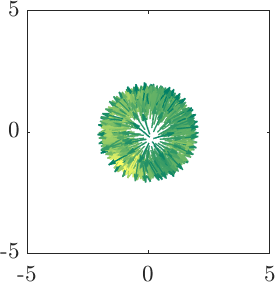}
			\caption{}
			\label{fig:vt_nc_10}
		\end{subfigure}
		~ 
		\begin{subfigure}[b]{0.33\textwidth}
			\includegraphics[scale=1]{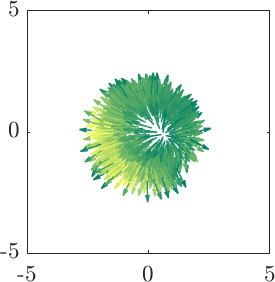}
			\caption{}
			\label{fig:vt_nc_20}
		\end{subfigure}
		\begin{subfigure}[b]{0.32\textwidth}
			\includegraphics[scale=1]{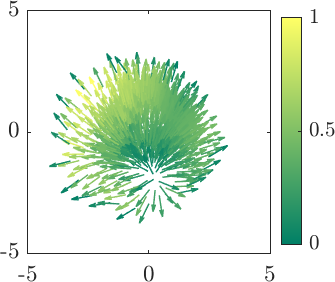}
			\caption{}
			\label{fig:vt_nc_30}
		\end{subfigure}
	\end{minipage}
}
	\caption{Set-NLM: Evolution of a tumour with a circular initial geometry. Rows one to four illustrate the variables volume fraction, nutrient concentration, negative pressure and cell momentum, respectively and columns one to three  illustrate state of the variables at times $T = 10,\,20$, and $30$, respectively.  }
	\label{fig:tumour_nlm}
\end{figure}